\documentclass{article}

\usepackage[utf8]{inputenc}
\usepackage[T1]{fontenc}
\usepackage{amsmath,amssymb,amsthm,bbm,float,graphicx,geometry,lmodern,mathtools,parskip,setspace,subcaption}
\usepackage[colorlinks, linkcolor = blue!80!black, citecolor = blue!80!black,breaklinks, pdfauthor={Benedikt Jahnel, Daniel Kamecke, Christof Külske}]{hyperref}
\usepackage{mathrsfs}
\usepackage{dsfont}
\usepackage{enumerate}
\usepackage{nicefrac}
\usepackage{todonotes}
\usepackage{comment}
\usepackage{titling}
\usepackage{fullpage}
\usepackage{comment}
\usepackage{orcidlink}
\usepackage{enumitem}
\usepackage[dvipsnames]{xcolor}

\usepackage{tikz}
\usetikzlibrary{backgrounds}
\usetikzlibrary{patterns}
\usetikzlibrary{positioning, shapes.geometric}

\usepackage{caption}
\usepackage{graphicx}
\graphicspath{{Images/}}
\allowdisplaybreaks

\newcommand{\N}{\mathbb{N}}

\newtheorem{theorem}{Theorem}[section]
\newtheorem{lemma}[theorem]{Lemma}

\newtheorem{prop}[theorem]{Proposition}

\theoremstyle{definition}

\usepackage{nameref}

\makeatletter
\let\orgdescriptionlabel\descriptionlabel
\renewcommand*{\descriptionlabel}[1]{%
  \let\orglabel\label
  \let\label\@gobble
  \phantomsection
  \edef\@currentlabel{#1}%
  \let\label\orglabel
  \orgdescriptionlabel{#1}%
}

\newcommand{\R}{\mathbb{R}}
\renewcommand{\N}{\mathbb{N}}

\newcommand{\E}{\mathbb{E}}

\newcommand{\Z}{\mathbb{Z}}

\newcommand{\W}{\mathcal{H}}
\renewcommand{\P}{\mathbb{P}}

\newcommand{\setslash}{\,|\, }

\newcommand{\compl}{\mathsf{c}}
\newcommand*\diff{\mathop{}\!\mathrm{d}}

\pretitle{\centering\LARGE\scshape}
 \posttitle{\vskip 0.75cm}

 \predate{\vskip 0.75 cm \centering\large}
 \postdate{\par}

\usepackage[style = numeric, sorting=nyt, url = true, abbreviate=false, maxbibnames=9, sortcites=true, doi = false, backend = biber, giveninits = true, isbn=false]{biblatex}
\renewbibmacro{in:}{\ifentrytype{article}{}{\printtext{\bibstring{in}\intitlepunct}}}
\bibliography{refs/refs.bib}

\title{Lattice random-field Widom--Rowlinson models}

\thanksmarkseries{arabic}

\author{
Benedikt Jahnel \orcidlink{0000-0002-4212-0065}\thanks{Technische Universit\"at Braunschweig, Universit\"atsplatz 2, 38106 Braunschweig, Germany}\thanksgap{0.4ex} 
\thanks{Weierstrass Institute for Applied Analysis and Stochastics, Mohrenstr.\ 39, 10117 Berlin, Germany}\\ benedikt.jahnel@tu-braunschweig.de
\and
Daniel Kamecke \orcidlink{}\thanksmark{1}\\
daniel.kamecke@tu-braunschweig.de
\and
Christof K\"ulske \orcidlink{}\thanks{Ruhr-Universit\"at Bochum, 
Fakult\"at f\"ur Mathematik, 44801 Bochum,
Germany}\\
christof.kuelske@ruhr-uni-bochum.de
}

\date{\today}

\begin{document}

\maketitle
\begin{spacing}{0.9}
\begin{abstract}  

We consider the Widom--Rowlinson model on $\Z^d$ subject to a symmetric i.i.d.\ random field. We prove that for dimensions $d\le 2$ any non-trivial random field leads to an absence of a phase transition. In contrast, in dimensions $d\ge 3$ and for Gaussian random fields, phase-transition behavior of the model is maintained for sufficiently large densities of occupied sites. This extends the general picture known from the classical random-field Ising model to the random-field Widom--Rowlinson model. Following the general proof route of Aizenman--Wehr as well as Ding--Zhuang, our main contribution rests on adequate notions of contours and their associated generalized spin-flip operation to deal with hard-core repulsions.

\medskip
\noindent\footnotesize{{\textbf{AMS-MSC 2020}: 60K37, 82B44}

}
\end{abstract}
\end{spacing}

\section{Introduction}\label{sec_intro}
Phase transitions in lattice systems with quenched disorder have long attracted interest in both mathematics and statistical physics. A paradigmatic example is the \emph{random‐field Ising model} (RFIM) with a famous heuristics given by Imry and Ma~\cite{imry1975random} and later rigorously studied in seminal works by Bricmont and Kupiainen~\cite{bricmont1988phase} as well as Aizenman and Wehr~\cite{aizenman1990rounding}.  In the RFIM one observes that an arbitrarily weak random field destroys long‐range order in dimensions \(d\le 2\) (the Imry--Ma phenomenon), while a true ordering phase persists for sufficiently small disorder when \(d\ge 3\). For an overview and refined results, see also~\cite{bovier2006statistical,aizenman2020exponential,ding2024long}.

In this paper, we investigate an analogous question for the \emph{Widom--Rowlinson model} (WRM) on \(\mathbb Z^d\) under the influence of an independent i.i.d.~random field.  In the classical WRM each site of the lattice can be in one of three states \(\{+1, -1, 0\}\), where \(+\) and \(-\) carry mutually repelling particles and \(0\) denotes vacancy.  On the lattice, the WRM exhibits a symmetry‐breaking phase transition between an ordered phase (favoring one of the two particle types) and a disordered (vacant–dominated) phase; this was rigorously established by Higuchi and Takei~\cite{HiTa04} via a random-cluster representation.

Introducing quenched randomness into the WRM by adding an independent random chemical potential \(\{\eta_x\}_{x\in\mathbb Z^d}\), i.i.d.\ with zero mean, raises a natural question: Does the prediction of the Imry--Ma argument extend to this hard-core setting? Recent progress on disordered hard‐core systems~\cite{ayuso2026imry}
shows that even for the hard‐core lattice gas, an arbitrarily weak random field suffices to destroy phase coexistence in two dimensions.  This suggests that similar destruction of order might occur in the two‐dimensional random‐field WRMs, while higher dimensions could support long‐range order for sufficiently weak disorder.

Considering the  \emph{random‐field Widom–Rowlinson model} (RFWRM) on \(\mathbb Z^d\) with i.i.d.~random field, we make the following contribution.
\begin{itemize}
  \item We prove the existence of at least two distinct extremal Gibbs measures for any centered i.i.d.\ Gaussian random field in dimension \(d\ge 3\),
  when the density of occupied sites is sufficiently large, see Theorem~\ref{Thm:Nonuniqueness}. Our approach adapts contour estimates as in~\cite{DiZh24} to the RFWRM.
  \item In dimension \(d\le 2\), we establish in Theorem~\ref{Thm:Uniqueness} the absence of symmetry breaking for arbitrary centered i.i.d.\ random fields using the strategy in~\cite{aizenman1990rounding}. 
\end{itemize}

Technically, the hard-core constraints in the WRM require us to implement delicate new notions of contours that allow to establish boundary-order energetic gains in dimensions \(d\ge 3\). 
This has to be complemented with a generalized flip operation that does not only exchanges 
particle signs and correspondingly random field signs, 
but also converts vacancies on the boundary of the flip set into particles. Our definition achieves an energetic gain 
while preserving the hard-core constraint in the flipped configuration. 
Here, we restrict ourself to the Gaussian case in order to avoid unnecessary further technicalities, but an extension of the arguments to a larger class of centered fields seems reasonable.
In the uniqueness case \(d\le2\), the hard-core constraint in the WRM introduces strong correlations that are not present in the RFIM and in the associated strategy in~\cite{aizenman1990rounding}. We overcome these additional challenges by a suitable removal strategy along the boundary. 

The remainder of the paper is organized as follows.  In Section~\ref{sec_main} we introduce the RFWRM and state our main results.  Section~\ref{sec_exist} establishes existence and stochastic monotonicity properties of the WRM in the presence of a fixed random field.  Sections~\ref{sec_nonuniqueness}~and~\ref{sec_uniqueness} are devoted to the proofs of the main statements. Finally, we provide additional details in our Appendix.

\section{Setting and main results}\label{sec_main}
Let $d\ge 1$ and consider the hypercubic lattice $\mathbb{Z}^d$ with nearest-neighbor edges.  
We denote by $\omega=(\omega_x)_{x\in\mathbb{Z}^d}\in\{-1,0,1\}^{\mathbb{Z}^d}\eqcolon \Omega$ a {\em Widom--Rowlinson configuration}, where $\omega_x=0$ represents an {\em empty site} and $\omega_x=\pm1$ represent the two {\em particle types}. For $\omega, \tilde \omega\in \Omega$, we define $\omega_\Lambda=(\omega_x)_{x\in \Lambda}$ and $\omega_\Lambda\tilde \omega_{\Lambda^\compl}\in \Omega$ as the concatenation of $\omega_\Lambda$ and $\tilde \omega_{\Lambda^\compl}$. Further, we call $\zeta =(\zeta_x)_{x\in \mathbb{Z}^d}\in \R^{\Z^d}=:\W$ a (deterministic and inhomogeneous) {\em external field configuration}. For $\Lambda\Subset\mathbb{Z}^d$, i.e., $\Lambda$ is finite, and {\em boundary configuration} $\tilde \omega\in\Omega$, the {\em finite-volume Hamiltonian} is defined as
\begin{align}\label{eq:Hamiltonian}
H^{\zeta,\lambda}_\Lambda(\omega)
 &= -\lambda\sum_{x\in\Lambda}\omega^2_x
    + \infty\hspace{-0.4cm}\sum_{x\in \Lambda,\, y\in \Z^d\colon x\sim y}
        \mathds{1}\{\omega_x\omega_y=-1\}
    - \sum_{x\in\Lambda}\zeta_x \omega_x,
\end{align}
where $\lambda\in \R$ allows us to tune the density of occupied sites and we say that $x\sim y$, if $x$ and $y$ are neighbors in~$\Z^d$, i.e., $\|x-y\|_1=1$ with $\|\cdot\|_1$ the $1$-norm. The {\em finite-volume Gibbs measure} in $\Lambda\Subset\Z^d$ with boundary condition $\tilde \omega$ and field $\zeta$ is defined as 
\[
\mu^{\zeta,\lambda}_{\Lambda,\tilde \omega}(A)
  = \frac{1}{Z^\zeta_{\Lambda,\tilde \omega}} 
    \sum_{\omega_\Lambda\colon \omega_\Lambda\tilde \omega_{\Lambda^\compl}\in A}\exp(-H^{\zeta,\lambda}_\Lambda\big(\omega_\Lambda\tilde \omega_{\Lambda^\compl})\big),\qquad \text{for measurable }A\subset \Omega,
\]
where $Z^\zeta_{\Lambda,\tilde \omega}$ is chosen so that $\mu^{\zeta,\lambda}_{\Lambda,\tilde \omega}$ is a probability measure. A probability measure $\mu^{\zeta,\lambda}$ on $\Omega$ is called a \emph{Gibbs measure of the Widom--Rowlinson model}, i.e., $\mu^{\zeta,\lambda}\in \mathcal{G}( \zeta,\lambda)$, if it satisfies the {\em DLR equations}, i.e., 
\begin{equation}\label{eq:DLR}
\int f(\omega) \, \mu^{\zeta,\lambda}({\rm d}\omega) 
  = \int\int f(\omega)\, \mu^{\zeta, \lambda}_{\Lambda,\tilde \omega}({\rm d}\omega)\mu^{\zeta,\lambda}({\rm d}\tilde \omega),\qquad \text{ for }\Lambda\Subset\mathbb{Z}^d\text{ and local } f\ge 0.
\end{equation} 
Here, a function $f\colon \Omega\to\R$ is called local if there exists $\Delta\Subset\Z^d$ such that $f(\omega)=f(\omega_\Delta)$ for all $\omega\in \Omega$.

Further, let $\eta=(\eta_x)_{x\in\mathbb{Z}^d}$ be an i.i.d.\ {\em random field} with values in $\mathcal{H}$ defined on a probability space $(\Psi,\mathcal{F},\mathbb{P})$. Throughout this work, we assume that
\begin{enumerate}[label=$(\roman*)$]
    \item\label{Ass:etaSymmetry} $\eta$ is symmetric, i.e.,~$\eta_o$ has the same distribution as $-\eta_o$, and
    \item \label{Ass:etaNontrivial}$\eta_o$ has finite and positive variance $0<\E[\eta_o^2]<\infty$.
\end{enumerate}
Our following main results show that the Gibbs measures of the {\em random-field WRM}, i.e.,~$\mu^{\eta, \lambda}\in \mathcal{G}(\eta, \lambda)$, behave similar to the Gibbs measures of the classical random-field Ising model in the sense that, in dimension $d=2$, the disorder $\eta$ destroys the otherwise observed ferromagnetic phase-transition behavior, whereas in higher dimensions small fields still allow for phase transitions.  
\begin{theorem}[Non-uniqueness]\label{Thm:Nonuniqueness}
Let $d\ge 3$ and assume that $\eta$ is a standard Gaussian field. Then, there exists $C>0$ such that for all $\varepsilon>0$ and $\lambda\ge C\max\{1,\varepsilon\}$ we have that $\mathbb{P}$–almost surely \(
|\mathcal{G}(\varepsilon\eta,\lambda)|>1 
\).
\end{theorem}

\begin{theorem}[Uniqueness]\label{Thm:Uniqueness}
Let $d\le 2$ and $\lambda\in \R$. Then, $\mathbb{P}$–almost surely
\(
|\mathcal{G}(\eta,\lambda)|=1 
\).
\end{theorem}

Here, and subsequently $|\cdot|$ denotes cardinality. Before we present the proofs in Sections~\ref{sec_nonuniqueness} and~\ref{sec_uniqueness}, we collect properties of existence and monotonicity of the WRM in the presence of an inhomogeneous field in the following section. 

\section{WRMs with fixed disorder}\label{sec_exist}
We focus on finite-volume Gibbs measures with {\em extremal boundary conditions}, i.e., $\smash{\mu^{\zeta,\lambda}_{\Lambda,\pm}}\coloneq\smash{\mu^{\zeta,\lambda}_{\Lambda,\mathbf{\pm 1}}} $ where $\mathbf{+1},\mathbf{-1},\mathbf{0}\in \Omega$ denote the configurations that are constant equal to $+1$, $-1$ and $0$. A key advantage in working with the WRM is that we can compare different Gibbs measures with the help of stochastic domination. In particular, this implies that the Gibbs measures $\smash{\mu^{\zeta,\lambda}_{\Lambda,\pm}}$ converge monotonically to their associated  extremal infinite-volume Gibbs measure $\smash{\mu^{\zeta,\lambda}_{\pm}}$. These are exactly the candidates for proving non-uniqueness of Gibbs measures in Section~\ref{sec_nonuniqueness}. Also towards a proof of uniqueness in Section~\ref{sec_uniqueness}, they come in handy as they stochastically dominate all Gibbs measures and hence it will suffice to prove $\mu^{\zeta,\lambda}_{+}=\mu^{\zeta,\lambda}_{-}$.  

To be precise, our configuration space $\Omega$ can be equipped with a partial order and we write $\omega \le \tilde \omega$ if $\omega_x\le \tilde \omega_x $ for all $x\in \Z^d$. Further, we say that $f\colon \Omega\rightarrow \R$ is \textit{increasing} if it is measurable and if for all $\omega, \tilde \omega \in \Omega$ with $\omega \le \tilde \omega$ it holds $f(\omega)\le f(\tilde \omega)$. For measures $\nu$ and $\tilde \nu$ on $\Omega$, we say that $\nu$ is \textit{stochastically dominated} by $\tilde \nu$, or $\nu\preceq\tilde \nu $, if $\nu(f)\le \tilde \nu(f)$ for every increasing and integrable $f$, where we use the notation $\nu(f)=\int_\Omega f(\omega )\, \nu(\diff \omega)$. Our first statement recovers well-known monotonicities of the WRM in the presence of inhomogeneous fields.
\begin{lemma}[Monotonicity]\label{Lemma:finiteVolumeFKG}
    For $\Lambda\Subset \Z^d$, $\zeta\in\W$ and $\omega,\tilde \omega\in \Omega$ with $\omega\le \tilde \omega$, it holds that $\mu_{\Lambda,\omega}^{\zeta,\lambda}\preceq\mu_{\Lambda,\tilde \omega}^{\zeta,\lambda}$. If further $\Delta\subset \Lambda$, it holds that $\smash{\mu_{\Lambda,+}^{\zeta,\lambda}}\preceq\smash{\mu_{\Delta,+}^{\zeta,\lambda}}$ and $\smash{\mu_{\Delta,-}^{\zeta,\lambda}}\preceq\smash{\mu_{ \Lambda,-}^{\zeta,\lambda}}$.
\end{lemma}
\begin{proof}
The proof is a minor generalization of the corresponding statement for homogeneous fields $\zeta\equiv h\in \R$, see for example~\cite[Lemma~2.3]{HiTa04}. Indeed, their proof consists in checking monotonicity of single-site conditional distributions under variation of  boundary conditions w.r.t.\ the partial order. As the external field $\zeta$ enters the single-site probability at $x$ only through its value at $x$, that is $\zeta_x$, the inhomogeneity of the field does not change the verification of monotonicity. Another part of their proof consists in checking certain properties of null-sets of $\mu_{\Lambda, \omega}^{\zeta,\lambda}$, which are not affected by an external field at all. 
\end{proof}

Next, we argue that the finite-volume Gibbs measures converge to infinite-volume Gibbs measures. We say that a sequence of measures $\mu_n$ on $\Omega$ \textit{converges locally} if $\nu_n(f)\xrightarrow{}\nu(f)$ for all local $f$.
\begin{prop}[Infinite-volume Gibbs measures]\label{Prop:Existence}
    Assume that $\Lambda_n\uparrow \Z^d$, i.e., $(\Lambda_n)_n$ is an increasing sequence of finite sets with $\bigcup_{n\in \N}\Lambda_n=\Z^d$. Then, for $\zeta \in\W$ and $\lambda\in \R$,
    \begin{enumerate}[label=(\roman*)]
        \item  there are $\mu_{\pm}^{\zeta,\lambda}\in \mathcal{G}(\zeta, \lambda)$ such that the sequence $\mu_{\Lambda_n,\pm }^{\zeta,\lambda}$ converges locally to $\mu_{\pm}^{\zeta,\lambda}$,
        \item for every $f$ local, $\zeta\mapsto \mu_{\pm}^{\zeta,\lambda}(f) $ is measurable when endowing $\W$ with the product-$\sigma$-algebra, and
        \item for every infinite-volume Gibbs measure $\nu^{\zeta,\lambda}$ of the WRM with external field $\zeta$, it holds that
        \begin{align*}
            \mu_{-}^{\zeta,\lambda}\preceq \nu ^{\zeta,\lambda}\preceq \mu_{+}^{\zeta,\lambda}.
        \end{align*}
    \end{enumerate}
\end{prop}
\begin{proof}
As for Item~$(i)$, by \cite[Proposition~4.9,~Proposition~4.15, and Theorem~4.17]{Ge11}, there exist a sequence $\Delta_n\uparrow \Z^d$ and an infinite-volume Gibbs measure $\smash{\mu_{\pm}^{\zeta,\lambda}}\in \mathcal{G}(\zeta, \lambda)$ such that $\smash{\mu_{ \Delta_n,\pm }^{\zeta,\lambda}}$ converges locally to $\smash{\mu_{\pm}^{\zeta,\lambda}}$. This local convergence can be generalized to arbitrary $\Lambda_n\uparrow \Z^d$. Indeed, we can find divergent sequences $(k_n)_n$ and $(\ell_n)_n$ in $\N$ with $\Delta_{k_n}\subset \Lambda_n\subset\Delta_{\ell_n}$ for all $n$. By Lemma~\ref{Lemma:finiteVolumeFKG}, for $f$ increasing and local, 
\begin{align*}
\mu_{\Delta_{\ell_n},+}^{\zeta,\lambda}(f)\le \mu_{\Lambda_n,+}^{\zeta,\lambda}(f) \le \mu_{\Delta_{k_n},+}^{\zeta,\lambda}(f).
\end{align*}
Since every local function can be represented as a linear combination of local and increasing functions, also $\smash{\mu_{\Lambda_n,\pm }^{\zeta,\lambda}}$ converges locally to $\smash{\mu_{\pm}^{\zeta,\lambda}}$.

As for Item~$(ii)$, by the previous convergence, for $f$ local, the mapping $\zeta\mapsto \smash{\mu_{+}^{\zeta,\lambda}(f)} $ is measurable as a limit of measurable functions. Finally, for Item~$(iii)$, by Lemma~\ref{Lemma:finiteVolumeFKG} and the DLR equations, for all local and increasing $f$ and all $n$ large enough,
    \begin{align*}
        \nu^\zeta(f)=\int_\Omega \mu^{\zeta,\lambda}_{\Lambda_n,\omega}(f)\, \nu^\zeta (\diff \omega)\le \mu_{\Lambda_n,+}^{\zeta,\lambda}(f),
    \end{align*}
    and analogously for $\mu_{\Lambda_n,-}^{\zeta,\lambda}$. The claim follows by sending $n$ to infinity.
\end{proof}

\section{Non-uniqueness in the RFWRM}\label{sec_nonuniqueness}
This section is devoted to the proof of Theorem~\ref{Thm:Nonuniqueness}, which is implied by the following more quantitative result. Assume that $\eta$ is a standard Gaussian field and consider the joint measure $$Q^\pm _{\Lambda,\varepsilon, \lambda}(A\times B)=\E[\mathds{1}_B(\eta) \mu_{\Lambda, \pm}^{\varepsilon\eta,\lambda}(A) ]$$
for $\varepsilon>0$ and $\Lambda\Subset\Z^d$. Let $\sigma\colon (\omega,\zeta)\mapsto \omega$ denote the canonical projection and $\sigma_o\colon (\omega,\zeta)\mapsto \omega_o$. 
\begin{theorem}[Magnetization bounds in joint measure]\label{Thm:3dPhaseTranstion}
There exist $c,C >0$ such that for all $\lambda\ge C\max\{1,\varepsilon\}$ and all $\Lambda\Subset \Z^d$
    \begin{align*}
        Q^{+}_{\Lambda,\varepsilon, \lambda}(\sigma_o=1)\ge 1-\mathrm{e}^{-c\lambda }-\mathrm{e}^{-c \lambda ^2/ \varepsilon^2 }.
    \end{align*}
\end{theorem}
This readily implies our first main result on non-uniqueness. 
\begin{proof}[Proof~of~Theorem~\ref{Thm:Nonuniqueness}]
By Proposition~\ref{Prop:Existence} and Theorem~\ref{Thm:3dPhaseTranstion} for sufficiently large $\lambda$, 
\begin{align*}
    \E\left[\mu_{+}^{\eta,\lambda}(\sigma_o=1)\right]>1/2>\E\left[\mu_{-}^{\eta,\lambda}(\sigma_o=1)\right],
\end{align*}
by symmetry. Therefore, with positive probability w.r.t.\ $\eta$ it holds that $\mu_{+}^{\eta,\lambda}(\sigma_o=1)\neq \mu_{-}^{\eta,\lambda}(\sigma_o=1)$ and hence $
|\mathcal{G}(\varepsilon\eta,\lambda)|>1 
$ with positive probability. Since the event  $\{\zeta\colon 
|\mathcal{G}(\varepsilon\zeta,\lambda)|>1 \}
$ is translation invariant and since $\eta$ is ergodic, we even have that $\P$-almost surely $
|\mathcal{G}(\varepsilon\eta,\lambda)|>1 
$.
\end{proof}
The remainder of this section is devoted to proving Theorem~\ref{Thm:3dPhaseTranstion}. In doing so, we follow the proof strategy for non-uniqueness of the random-field Ising model (RFIM) presented in \cite{DiZh24}. In the Ising model there are only two spin values $\sigma_o=1$ or $\sigma_o=-1$ and a boundary of surface area $n$, also called contour, between two regions with opposite spin values causes an energetic cost of order $n$. When searching for the desired upper bound on $Q^{+}_{\Lambda,\varepsilon, \lambda}(\sigma_o\neq 1)$, \cite{DiZh24} utilizes this energetic cost of order $n$ by finding the smallest contour surrounding the origin $\sigma_o=-1$ and then flipping all the signs of $\sigma$ and $\eta$ that are in its interior. The joint sign flip of $\sigma$ and $\eta$ causes an energetic gain of order $n$, which dominates the influence of the random field $\eta$ with large probability w.r.t.\ the disorder.  
In this way the same type of long-range order as in the Ising model without random field is maintained. A key contribution of our work lies in finding a suitable analogue of the notion of a boundary of the region where $\sigma_o\neq 1 $ around the origin, as well as an associated joint spin flip so that there is an useful energetic gain also in the RFWRM. More precisely, we will see that in our case the flip operation 
does not only interchange minuses and pluses, but also converts zeros at the boundary into pluses. In order to avoid violation of the hard-core constraint after application of the flip operation, we work with thickened boundaries. To make this rigorous, we introduce the following notions of boundary and connectedness with respect to $m\in\N $.  
\begin{enumerate}[label=(\roman*)]
    \item For $x,y\in \Z^d$, we say $x\overset{m}{\sim} y$ if $\|x-y\|_1\le m$.
    \item  A {\em path} $\pi\colon \Z_{\ge 0}\rightarrow \Z^d$ is {\em $m$-connected} if  $\pi(t)\overset{m}{\sim}\pi(t+1)$ for all $t$.
    \item A set $A\subset \Z^d$ is {\em $m$-connected} if there is, for every $x,y\in A$, an $m$-connected path $\pi$ with values in $A$ and such that $\pi(0)=x$ and $\pi(t)=y$ for some $t\in \Z_{\ge 0}$.
    \item A set $A\subset \Z^d$ is {\em $m$-simply connected} if both $A$ and $A^\compl$ are $m$-connected. 
    \item We say that $x\in \Z^d$ is {\em $m$-enclosed} by some set $A\subset \Z^d$ if for every $m$-connected and self-avoiding path $\pi\colon \Z_{\ge 0}\rightarrow \Z^d$ with $\pi(0)=x$ there is $t\in \Z_{\ge 0}$ with $\pi(t)\in A$.
    \item The (inner) {\em $m$-boundary} of $A$ is defined as 
    \begin{align*}
        \partial^m A= \{x \in A\setslash\exists y\in A^\compl \colon x\overset{m}{\sim}y\}.
    \end{align*}
    \item Let $\mathfrak{A}_m$ be the space of all $m$-simply connected and finite subsets of $\Z^d$ containing the origin.
\end{enumerate} 
Next, we want to define for every $\omega \in \Omega$ an $A_\omega\in \mathfrak{A}_2$ whose boundary $\partial ^1 A_\omega$ can be interpreted as the contour surrounding the origin. If $\omega_o=1$, set $A_\omega =\emptyset$. If $\omega_o\neq 1$ set $B_\omega$ as the maximal $1$-connected set containing the origin with $\omega_x\neq 1$ for all $x\in B_\omega$. For every $A\subset \Z^d$ we define the $2$-closure of $A$ as
\begin{align*}
    \bar A=\{x\in \Z^d\setslash x \text{ is $2$-enclosed by }A \}.
\end{align*} Note that $A\subset \bar A$. Still in the case $\omega_o\neq 1$, define $A_\omega=\bar B_\omega$. For a visualization of $A_\omega$, see the left picture of Figure~\ref{fig:ConfigurationFlipped}. Here, $B_\omega$ contains the blue and white vertices surrounding the origin and $A_\omega$, indicated by the black lines, includes also the red vertices that are surrounded by $B_\omega$ in the $2$-enclosed sense. Note that two of the red vertices are not included in $A_\omega$ as they are only $1$-enclosed by $B_\omega$ and not $2$-enclosed. 

The following lemma hints as to why we define $A_\omega$ in this fashion. More precisely, the first reason why it is useful is that on $\partial^1 A_\omega$ there are only zeros
when restricting to local feasible configurations with extremal boundary, i.e., 
\begin{align*}
    \Omega^+ _\Lambda=\{\omega\in \Omega\setslash \omega_{\Lambda^\compl}=(\mathbf{+1})_{\Lambda^\compl},\    H_\Lambda^{\mathbf{0},\lambda}(\omega)<\infty \},\qquad o\in \Lambda\Subset\Z^d.
\end{align*}
Second, we show a control on the number of sets $A$ that can be obtained as $A=A_\omega$. 
\begin{lemma}[Contour properties]\label{Lemma:CountingSimpleSets}
We have the following contour properties for $d\ge 2$ and finite $\Lambda\in  \mathfrak{A}_2$. 
\begin{enumerate}[label=(\roman*)]
    \item For every $\omega \in \Omega^+_\Lambda$ with $\omega_o\neq 1$, it holds that $A_\omega\in \mathfrak{A}_2$ and $\omega_{\partial ^1A_\omega} =\mathbf{0}$. 
    \item There exists $C>0$ such that for all $n\in \N$
    \begin{align*}
    |\{A\in \mathfrak{A}_2\setslash |  \partial ^2 A|=n \}|\le C^{n}.
    \end{align*}
\end{enumerate}
\end{lemma}
This lemma can be proved by connectivity and counting arguments. For the reader's convenience, a detailed proof is included in Appendix~\ref{Sec:Appendix_Boundary}. The third reason why $A_\omega$ is useful becomes apparent when defining a generalized spin flip of $\omega \in \Omega$ and $\zeta\in \W$ on $A\Subset \Z^d$, i.e., 
\begin{align*}
    \omega ^A_x=\begin{cases}1  &x\in \partial^1 A\\
    -\omega _x &x\in A\setminus \partial^1 A \\
        \omega_x &x \in A ^{\compl }
        \end{cases}    
         \qquad \text{and}\qquad \zeta ^A_x=\begin{cases}
        -\zeta_x &x\in A\\
        \zeta_x &x \in A ^\compl.
    \end{cases}
\end{align*}
For a visualization of $\omega^{A_\omega}$, see Figure~\ref{fig:ConfigurationFlipped}.
\begin{figure}
    \centering
    \includegraphics[width=\linewidth]{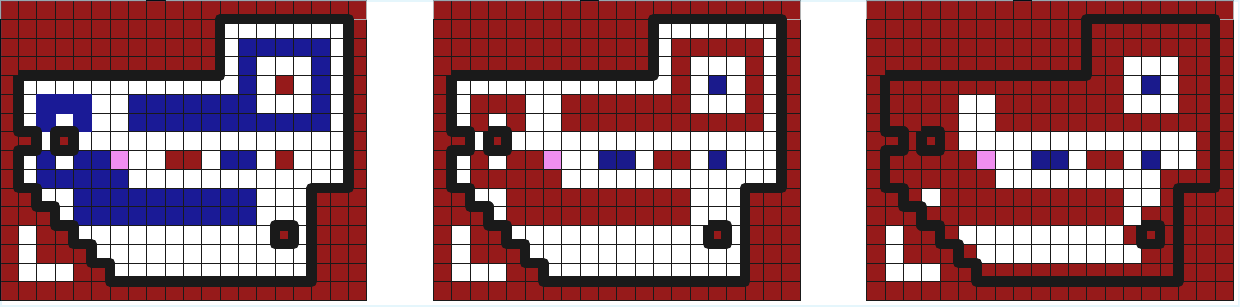}
    \caption{A visualization of a WRM configuration, where blue represents $-1$, red $1$, white $0$ and violet the origin.  In all three pictures, the black lines surround $A_\omega$. The left picture represents $\omega$ which is $-1$ in the origin. In the middle picture we flip $\omega_x$ to $-\omega_x$ for $x\in A_\omega\setminus \partial^1 A_\omega$. On the right is the result of the generalized flip $\omega^{A_\omega}$, i.e., where also $\omega_x=1$ for $x\in \partial ^1 A_\omega$. While flipping the inner boundary to $1$ will ensure an energetic gain of boundary order, the choice of $A_\omega$ and the spin flip inside of $A_\omega$ is made to satisfy the hard-core condition of the WRM. Note that the isolated red vertex on the left side of the pictures is not flipped as it is not $2$-enclosed by $B_\omega$.}
    \label{fig:ConfigurationFlipped}
\end{figure}
Write $f^+_\Lambda\colon \{-1,0,1\}^{\Lambda}\times \R^{\Lambda}$ as the density of $Q^+_{\Lambda,\varepsilon, \lambda}$ w.r.t.\ the natural choice of measure on $\{-1,0,1\}^{\Lambda}\times \R^{\Lambda}$, i.e., 
\begin{align*}
    f^+_\Lambda(\omega, \zeta)=\frac{{\rm e}^{-H^{\varepsilon\zeta, \lambda} _\Lambda(\omega _\Lambda (\mathbf{+1})_{\Lambda^{\compl}})}}{Z^{\varepsilon\zeta} _{\Lambda, +}}\prod_{x\in \Lambda}\frac{1}{\sqrt{2\pi }}{\rm e}^{-\zeta_x^2/2},
\end{align*}
where $Z^{\varepsilon\zeta} _{\Lambda, +}=Z^{\varepsilon\zeta} _{\Lambda,\mathbf{+1}}$.
The generalized spin-flip operation $(\omega,\zeta)\mapsto (\omega ^{A_\omega},\zeta^{A_\omega})$ guaranties the following when restricting to $\Omega^+_\Lambda$.
\begin{lemma}[Spin-flip properties]\label{Lemma:DensityQuotient}
Let $o\in\Lambda\Subset\Z^d$, $\omega \in \Omega^+_\Lambda$ with $\omega_o\neq 1$, and $\zeta\in \R^{\Lambda}$. Then, 
\begin{enumerate}[label=(\roman*)]
    \item the flipped configuration obeys the hard-core constraint $\omega ^{A_\omega }\in \Omega^+_\Lambda$ and
    \item the quotient of the densities is given by
    \begin{align*}
    \frac{f^+_\Lambda(\omega, \zeta)}{f^+_\Lambda(\omega^{A_\omega}, \zeta^{A_\omega})}=\exp\Big(\sum_{x\in  \partial^1 A_\omega }(-\lambda +\varepsilon \zeta_x) 
    \Big)Z^{\varepsilon \zeta^{A_\omega }} _{\Lambda, + }/Z^{\varepsilon\zeta} _{\Lambda, + }.
    \end{align*}
\end{enumerate}
\end{lemma}
\begin{proof}
For Item~$(i)$, by definition of $\omega^{A_\omega}$, it holds that $\omega^{A_\omega}_{\Lambda^\compl}=(\mathbf{+1})_{\Lambda^\compl}$. Since $H_\Lambda^{\mathbf{0},\lambda}(\omega)<\infty$, we know that the hard-core condition is satisfied, i.e., for all $x,y\in \Z^d$ 
\begin{align*}
     x\overset{1}{\sim} y\implies   \omega_x\omega_y\neq -1.
\end{align*}
Next, we need to show that this implies that also $\omega^{A_\omega}$ does not violate the hard-core condition, i.e., that for all $x,y\in \Z^d$,
\begin{align}\label{Eq:HardcoreImplication}
    x\overset{1}{\sim} y\implies  \omega_x^{A_\omega}\omega^{A_\omega}_y\neq -1.
\end{align}
We prove \eqref{Eq:HardcoreImplication} by checking the complete list of cases. 
\begin{enumerate}[label=(\roman*)]
    \item If $x,y\in A_\omega ^\compl$ or $x,y\in A_\omega\setminus \partial ^1 A_\omega$ with $ x\,\smash{\overset{1}{\sim}}\,  y$, Condition~\eqref{Eq:HardcoreImplication} follows since 
$
\omega_x^{A_\omega}\omega^{A_\omega}_y=\omega_x\omega_y\neq -1.
$

    \item If $x\in A_\omega ^\compl $ and $y\in A_\omega \setminus \partial^1 A_\omega$, Condition~\eqref{Eq:HardcoreImplication} is satisfied since $x\, \smash{\overset{1}{\sim}}\,  y$ never occurs by the definition of the boundary $\partial^1 A_\omega$.
    \item Also for $x\in \partial^1 A_\omega  $ and $y\in A_\omega\setminus B_\omega $ the condition $x\, \smash{\overset{1}{\sim}}\,  y$ never happens. At this point, it is crucial that we used the $2$-closure of $B_\omega$ and not simply the $1$-closure. Due to this fact, one can construct a path from $y$ over $x$ in the $1$-boundary of $A_\omega$ to $\smash{A_\omega ^\compl}$, which in turn contradicts the fact that $y$ is $2$-enclosed by $B_\omega$. More precisely, since $x\in \partial^1 A_\omega$, there exists an element $\tilde x\in \smash{A_\omega^\compl}$ with $\tilde x\, \smash{\overset{1}{\sim}}\,  x$. By definition $\tilde x$ is not $2$-enclosed by $B_\omega$, which means that there is an infinite, self-avoiding, and $2$-connected path starting in $\tilde x$ and staying in $B_\omega^\compl$. The assumption $x\, \smash{\overset{1}{\sim}}\,  y$ together with $\tilde x\, \smash{\overset{1}{\sim}}\,  x$ leads to $y\, \smash{\overset{2}{\sim}}\,  \tilde x$ and we can construct also an infinite, self-avoiding, and $2$-connected path starting in $y$ and staying in $B_\omega^\compl$. But this is a contradiction to $y\in A_\omega\setminus B_\omega $, since this means that $y$ is $2$-enclosed by $B_\omega$. 
    
    Let us remark that this case explains why we needed to exclude the isolated red vertex on the left side of the pictures in Figure~\ref{fig:ConfigurationFlipped} from $A_\omega$ by choosing $A_\omega$ as the $2$-enclosure of $B_\omega$ and not the $1$-enclosure.
    \item If $x\in \partial^1 A_\omega  $ and $y\in B_\omega $, it holds that $\omega ^{A_\omega }_x=1$ and $\omega^{A_\omega}_y=-\omega_y\neq -1$. Overall, $\omega ^{A_\omega }_x\omega^{A_\omega}_y\neq -1$ and Condition~\eqref{Eq:HardcoreImplication} is satisfied.
    \item If $x\in \partial^1  A_\omega $ and $y\in A_\omega ^\compl $ with $x\, \smash{\overset{1}{\sim}}\,  y$, we have $\omega_x^{A_\omega}= 1$ and by maximality of $B_\omega$ also $\omega^{A_\omega}_y=\omega_y=1$.
\end{enumerate}
    Hence, Condition~\eqref{Eq:HardcoreImplication} holds for all $x,y\in \Z^d$ and thus, $H_\Lambda^{\mathbf{0},\lambda}(\omega^{A_\omega})<\infty$ for all $\omega\in \Omega^+_\Lambda$. For Item~$(ii)$, we write
    \begin{align*}
    \frac{f^+_\Lambda(\omega, \zeta)}{f^+_\Lambda(\omega^{A_\omega}, \zeta^{A _\omega})}&=\exp\left(-H^{\varepsilon\zeta, \lambda} _\Lambda(\omega _\Lambda (\mathbf{+1})_{\Lambda^{\compl}})+H^{\varepsilon\zeta ^{A_\omega},\lambda} _\Lambda(\omega^{A_\omega} _\Lambda (\mathbf{+1})_{\Lambda^{\compl}})\right)Z^{\varepsilon \zeta^{A_\omega }} _{\Lambda, + }/Z^{\varepsilon\zeta} _{\Lambda, + }\\
    &=\exp\Big(\lambda \sum_{x\in\Lambda}\omega^2_x
    + \sum_{x\in\Lambda}\varepsilon \zeta_x \omega_x- \lambda \sum_{x\in\Lambda}\Big(\omega_x^{A_\omega }\Big)^2
    -\sum_{x\in\Lambda}\varepsilon \zeta_x^{A_\omega} \omega_x^{A_\omega}\Big)Z^{\varepsilon \zeta^{A_\omega }} _{\Lambda, + }/Z^{\varepsilon\zeta} _{\Lambda, + }.
\end{align*}
By Lemma~\ref{Lemma:CountingSimpleSets} we know that $\omega_{\partial^1 A_\omega}=\mathbf{0}$ as well as $\omega_{\partial^1 A_\omega}^{A_\omega}=\mathbf{1}$ and hence
\begin{align*}
    \frac{f^+_\Lambda(\omega, \zeta)}{f^+_\Lambda(\omega^{A_\omega}, \zeta^{A _\omega})}&=\exp\Big( \sum_{x\in  \partial ^1 A_\omega }(-\lambda + \varepsilon \zeta_x) 
\Big)Z^{\varepsilon \zeta^{A_\omega }} _{\Lambda, + }/Z^{\varepsilon\zeta} _{\Lambda, + },
\end{align*}
as desired. 
\end{proof}
To summarize, when using the joint flip $(\omega,\zeta)\mapsto (\omega^{A_\omega},\zeta^{A_\omega})$, we get an energetic gain of $\lambda |\partial^1 A_\omega|$, which is competing with the field terms coming from $\zeta$. Let us come back to the arguments of \cite{DiZh24} to see how to utilize this energetic gain in order to find the desired upper bound on $Q_{\Lambda,\varepsilon,\lambda}^+(\sigma_o\neq 1)$ as stated in Theorem~\ref{Thm:3dPhaseTranstion}. For all $E^\star \subset \W$ measurable, we can write by Lemmas~\ref{Lemma:CountingSimpleSets}~and~\ref{Lemma:DensityQuotient},
\begin{align*}
    Q_{\Lambda,\varepsilon,\lambda}^+(\sigma_o\neq 1)&\le Q_{\Lambda,\varepsilon,\lambda}^+(\sigma_o\neq 1,\eta \in E^\star)+Q_{\Lambda,\varepsilon,\lambda}^+(\eta \notin E^\star)\\
    &\le\sum_{A\in \mathfrak{A}_2} \int _{E^\star} \sum_{\omega\in \Omega^+_\Lambda\colon A_\omega= A}f^{+}_\Lambda (\omega , \zeta)\diff \zeta +\P(\eta \notin E^\star) \\
    &=\sum_{A\in \mathfrak{A}_2} \int _{E^\star} \sum_{\omega\in \Omega_\Lambda^+ \colon A_\omega= A}f^+_\Lambda(\omega^{A}, \zeta^{A})\exp\Big( \sum_{x\in  \partial ^1 A}(-\lambda + \varepsilon \zeta_x) 
\Big)Z^{\varepsilon\zeta ^{A}} _{\Lambda, + }/Z^{\varepsilon\zeta} _{\Lambda, + }\diff \zeta+\P(\eta \notin E^\star).
\end{align*}
Both terms $Z^{\varepsilon\zeta^{A}}  _{\Lambda, + }/Z^{\varepsilon\zeta} _{\Lambda, + }$ and $\exp\big(\sum_{x\in \partial^1 A}\varepsilon\zeta_x\big)$ can be arbitrarily large when not controlled via $\zeta\in E^\star$. For this, we define $E^\star =E_{\Lambda,\lambda/4,\varepsilon}\cap \tilde E_{\lambda/4,\varepsilon }$, where for all $z$ and $\varepsilon$
\begin{align*}
     E_{\Lambda, z,\varepsilon}&\coloneq\big\{\zeta\in \mathcal{H} \setslash\forall A\in \mathfrak{A}_2\colon Z^{\varepsilon \zeta ^A} _{\Lambda, + }/ Z^{\varepsilon \zeta } _{\Lambda, + }\le \exp(z|\partial^1 A|)\big\}\quad \text{ and }\\
    \tilde E_{z,\varepsilon}&\coloneq\Big\{\zeta\in \mathcal{H} \setslash\forall A\in \mathfrak{A}_2\colon \sum_{x\in \partial^1 A}\varepsilon\zeta_x\le z |\partial^1 A| \Big\}.
\end{align*}
Then, 
\begin{align*}
     Q_{\Lambda,\varepsilon,\lambda}^+(\sigma_o\neq 1)\le\sum_{A\in \mathfrak{A}_2} \exp\Big(- \lambda| \partial ^1 A|/2 
\Big)\int _{E^\star} \sum_{\omega\in \Omega_\Lambda^+ \colon A_\omega= A}f^+_\Lambda(\omega^{A}, \zeta^{A})\diff \zeta+ \P(\eta \notin E^\star).
\end{align*}
Define for $\omega\in \Omega_{\Lambda}^+$ the inverse flip by $\iota_A( \omega)  _x=0$ for $x\in \partial^1 A$, $\iota_A(\omega)_x=-\omega_x$ for $x\in A\setminus \partial ^1A$ and $\iota_A( \omega) _x=\omega_x$ else. This is indeed the inverse flip in the sense that it is the inverse of 
\begin{align*}
    \{\omega\in \Omega_\Lambda^+\setslash \omega_{\partial ^1 A}=\mathbf{0}\}\rightarrow \{\omega\in \Omega_\Lambda^+\setslash \omega_{\partial ^1 A}=\mathbf{1}\} ,\quad\omega \mapsto \omega ^A .
\end{align*}
Due to this bijectivity, we can write
\begin{align*}
    \int  \sum_{\omega\colon A_\omega= A}f^{+}_\Lambda (\omega ^A , \zeta^A)\diff \zeta=\int  \sum_{\omega\colon A_{\iota_A(\omega )}= A,\,  \omega_{\partial^1 A}=\mathbf{1}}f^{+}_\Lambda (\omega  , \zeta)\diff \zeta  \le\int \sum_{\omega\in \Omega_\Lambda^+}f^{+}_\Lambda (\omega  , \zeta)\diff \zeta=1.  
\end{align*}
By Lemmas~\ref{Lemma:CountingSimpleSets}~and~\ref{Lemma:Equivalence_Boundaries}, there are $c,C>0$ such that
\begin{align*}
     Q_{\Lambda,\varepsilon,\lambda}^+(\sigma_o\neq 1)&\le\sum_{A\in \mathfrak{A}_2} \exp\Big(-  \lambda| \partial ^1 A|/2\Big)+ \P(\eta \notin E^\star)\\
     &\le \sum_{A\in \mathfrak{A}_2} \exp\Big(-  c\lambda| \partial ^2 A|\Big)+ \P(\eta \notin E^\star)\\
     &\le\sum_{n=1}^\infty C^n \exp(-c\lambda n  )+\P(\eta \notin E^\star).
\end{align*}
Therefore, there are $\tilde c,\tilde C>0$ such that for  $\lambda\ge \tilde C$
\begin{align}\label{eq:phasetransitionAlmostthere}
     Q_{\Lambda,\varepsilon,\lambda}^+(\sigma_o\neq 1)\le \mathrm{e}^{-\tilde c\lambda}+\P(\eta \notin E_{\Lambda, \lambda/4,\varepsilon})+\P(\eta \notin \tilde E_{ \lambda/4,\varepsilon}).
\end{align}
Note that in the absence of an external field, i.e., $\varepsilon=0$, the proof steps developed so far imply the existence of the phase transition in the classical WRM, even for $d=2$, as then $\P(\eta \notin E^\star)=0$. In the RFWRM with $\varepsilon>0$ we are finished when bounding the last two terms in~\eqref{eq:phasetransitionAlmostthere} appropriately, where the latter term is the easier one. In particular this will explain why for the phase transition $\lambda/\varepsilon$ has to be large.
\begin{prop}[Gaussian concentration]\label{Prop:Etildeillbehaved}
There are $c,C >0$ such that for all $z/\varepsilon>C$, it holds that
\begin{align*}
    \P(\eta \notin \tilde E_{ z,\varepsilon})\le \mathrm{e}^{-cz^2 /\varepsilon^2}.
\end{align*}

\end{prop}
\begin{proof}
For $N$ a standard Gaussian there exists $C'>0$ such that 
\begin{align*}
    \P(\eta \notin \tilde E_{z,\varepsilon} )\le \sum_{A\in \mathfrak{A}_2}\P(N\ge \varepsilon^{-1}z\sqrt{|\partial ^1 A|})\le C'\sum_{A\in \mathfrak{A}_2}\exp(-|\partial ^1 A|z^2/(2\varepsilon^2)).
    \end{align*}

By Lemma~\ref{Lemma:CountingSimpleSets}, there is $\tilde C>0$ such that
\begin{align*}
    \P(\eta \notin \tilde E_{ z,\varepsilon})\le C'\sum_{n=1}^\infty \tilde C^n \exp(-nz^2/(2\varepsilon^2)),
\end{align*}
which proves the claim. 
\end{proof}
To arrive at a bound for $\P(\eta \notin E_{\Lambda, z,\varepsilon})$, we use a rather general coarse-graining argument first derived in \cite{FiFrSp84}. Let $A\ominus B = (A \setminus B )\cup (B \setminus A)$ denote the symmetric difference of two sets.
\begin{prop}[Uniform boundary-order concentration]\label{Prop:CoarseGraining}
    Let $d\ge 3$. For every $c,C>0$, there are $\tilde c, \tilde C>0$ such that if a family of random variables $(F_{A,\varepsilon})_{A\Subset \Z^d,\varepsilon}$ w.r.t.\ $(\Psi, \mathcal{F}, \P)$
    satisfies for all $z,\varepsilon>0$ 
    and $A,B \Subset \Z^d$
    \begin{align}\label{Eq:PropertyF}
         \P(|F_{A,\varepsilon} |\ge z)\le  C\exp\Big(-\frac{cz^2}{\varepsilon^2 |A|}\Big)\quad \text{and}\quad \P(|F_{A,\varepsilon} -F_{B,\varepsilon}|\ge z)\le C\exp\Big(-\frac{cz^2}{\varepsilon^2 |A  \ominus B|}\Big),
    \end{align}
     we have that for all $z,\varepsilon>0$ and $m=1,2$ 
     \begin{align*}
         \P(\exists A\in \mathfrak{A}_m\colon |F_{A,\varepsilon} |\ge z |\partial^1 A|)\le  \tilde C \exp(- \tilde c z^2/\varepsilon^2).
     \end{align*}
\end{prop}
Note that in \cite{FiFrSp84}, this result is stated for one specific choice of $(F_{A,\varepsilon})_{A,\varepsilon}$, $m=1$ and $d=3$ but the proof relies only on the Property~\eqref{Eq:PropertyF} as pointed out in \cite{DiZh24}. In particular, the constants $\tilde c, \tilde C$ are only dependent on the dimension $d$ and on $c,C>0$ but not on the choice of $(F_{A,\varepsilon})_{A,\varepsilon}$ or on $z$ or $\varepsilon$. For the reader's convenience, we added a stand-alone proof of Proposition~\ref{Prop:CoarseGraining} in Appendix~\ref{Sec:AppCoarseGraining} to account for the generality needed here. We apply the result above to families of random variables $(F_{A,\varepsilon})$ that can be written as Lipschitz functions of independent random variables. We call $f\colon \R^n\rightarrow \R$ Lipschitz continuous if there exists a Lipschitz constant $L>0$ such that for all $x,y\in \R^n$
\begin{align*}
    |f(x)-f(y)|\le L\|x-y\|_2,
\end{align*}
where $\|\cdot\|_2$ is the Euclidean norm. Here, the key Property~\eqref{Eq:PropertyF} can be reduced to the following concentration inequality.

\begin{prop}[Concentration for Lipschitz functions]\label{Prop:Concentration_Inequality}
 There exist constants $c,C>0$ such that for all $n\in \N$ the following is true. 
Assume that $X=(X_1,\ldots, X_n)$ are independent standard Gaussian random variables and  $f\colon \R^n\rightarrow \R$  Lipschitz continuous with Lipschitz constant $L>0$. Then, 
\begin{align*}
\P(|f(X)-\E[f(X)]|\ge z )    \le C\exp\left(-cz^2/L^2\right),\qquad \text{ for all }z>0.
\end{align*}
\end{prop}
\begin{proof}
First, note that we can replace $\E[f(X)]$ in the claim above with the median of $f(X)$, see for example~\cite[Exercise~5.6, Proposition~2.6.1]{Ve18}. The statement for the Gaussian case with the median instead of the expected value is proved, e.g., in \cite[Chapter~1~Eq.~(1.4)]{LeTa91}. 
\end{proof}
The random variable $X$ will play the role of $\eta$. This concentration inequality is the sole reason that we restrict ourselves to Gaussian $\eta$. Note that there are also concentration inequalities for other classes of random variables, such as bounded random variables, see for example~\cite[Theorem~6.6]{Ta96}. However, these require convex level sets of $f$, which seem cumbersome to check in our application. 
We apply both of the above propositions to our setting as follows. 
\begin{prop}[Concentration for partition functions]\label{Prop:partitionFunctionillbehaved}There are $c,C>0$ such that for all $z/\varepsilon>C$ and $\Lambda \Subset \R^d$, it holds that
\begin{align*}
    \P(\eta \notin  E_{\Lambda,z,\varepsilon})\le \mathrm{e}^{-cz^2  /\varepsilon^2}.
\end{align*}

\end{prop}
\begin{proof}
When defining the random variables
    \begin{align}\label{eq:FLambda}
        F_{A,\varepsilon}\coloneq \log\Big(Z^{\varepsilon\eta^A}_{\Lambda,+}\Big)-\log\Big(Z^{\varepsilon \eta}_{\Lambda,+}\Big),
    \end{align}
    it holds that
    \begin{align*}
    \P(\eta \notin E_{\Lambda,z,\varepsilon})\le \P(\exists A\in \mathfrak{A}_2\colon |F_{A,\varepsilon}|\ge z|\partial^1 A|).
    \end{align*}
    So, by Proposition~\ref{Prop:CoarseGraining} it suffices to prove \eqref{Eq:PropertyF}. For fixed $\Lambda \Subset \Z^d$, $\tilde \zeta\in\W$ and $A\subset \Z^d$, define the function $f_{\tilde \zeta,A}\colon \R^{A}\rightarrow \R$ by 
    \begin{align*}
        f_{\tilde \zeta,A}(\zeta)=\log\Big(Z^{\varepsilon\tilde \zeta_{A^\compl}(-\zeta) _A}_{\Lambda,+}\Big)-\log\Big(Z^{\varepsilon\tilde \zeta _{A^\compl}\zeta_A}_{\Lambda,+}\Big).
    \end{align*}
    Then, the partial derivative in the $x$th component for $x\in A$ of $f_{\tilde \zeta ,A}$ is
    \begin{align*}
        |\partial _x f_{\tilde \zeta,A}(\zeta)|&= \Big|\varepsilon\int \omega_x^A\mu^{\varepsilon\tilde \zeta_{A^\compl}(-\zeta) _A}_{\Lambda,+}(\diff \omega)+\varepsilon\int \omega_x^A\mu^{\varepsilon\tilde \zeta_{A^\compl}\zeta _A}_{\Lambda,+}(\diff \omega)\Big|\le 2\varepsilon.
    \end{align*}
    Hence, by the Cauchy--Schwarz inequality
    \begin{align*}
    |f_{\tilde \zeta, A}(\zeta)-f_{\tilde \zeta, A}(\zeta^\prime)|\le \sum_{x\in A}2\varepsilon |\zeta_x- \zeta_x^\prime|\le 2 \varepsilon \sqrt{|A|}\|\zeta-\zeta^\prime\|_2.
    \end{align*}
     Further, by symmetry of $\eta$, we have that
$
        \E[f_{\tilde \zeta,A}(\eta)]=0
$
    and by Proposition~\ref{Prop:Concentration_Inequality} there are constants $c,C>0$, independent of $\Lambda$, $\tilde \zeta$, $\varepsilon$ and $A$, such that for all $z>0$
    \begin{align*}
    \P(|f_{\tilde \zeta,A}(\eta)|\ge z)\le C\exp\left(-\frac{cz^2}{|A|\varepsilon^2}\right).
    \end{align*}
    Let $\tilde \eta$ be an independent copy of $\eta$ with associated expectation $\tilde \E$. Then, 
    \begin{align}\label{Eq:ConcentrationInequalityF}
                \P(|F_{A,\varepsilon}|\ge z)=\P(|f_{\eta,A}(\eta )|\ge z)=\tilde \E[\P(|f_{\tilde \eta,A}(\eta )|\ge z)]\le C\exp\left(-\frac{cz^2}{|A|\varepsilon^2}\right),
    \end{align}
    which is the first inequality in \eqref{Eq:PropertyF}. Note that $(\eta^A,\eta^B)$ has the same distribution as $(\eta^{A\ominus B},\eta )$, since the distribution of the latter can be obtained as the former after applying the spin flip on $B$ due to symmetry of $\eta$. As a result, we can apply \eqref{Eq:ConcentrationInequalityF} also to $A\ominus B$ instead of $A$ and get 
    \begin{align*}
        \P(|F_{A,\varepsilon}-F_{B,\varepsilon}|\ge z)&=\P(|f_{\eta,A}(\eta )-f_{\eta,B}(\eta )|\ge z)=\P(|f_{\eta,A\ominus B}(\eta )|\ge z)\le C\exp\left(-\frac{cz^2}{|A\ominus B|\varepsilon^2}\right).
    \end{align*}
     This proves \eqref{Eq:PropertyF} and hence the claim.
\end{proof}
Finally, we are ready to prove Theorem~\ref{Thm:3dPhaseTranstion}.
\begin{proof}[Proof~of~Theorem~\ref{Thm:3dPhaseTranstion}]
The claim follows by Equation~\eqref{eq:phasetransitionAlmostthere}, Proposition~\ref{Prop:Etildeillbehaved} and Proposition~\ref{Prop:partitionFunctionillbehaved}. 
\end{proof}
\section{Uniqueness in the RFWRM}\label{sec_uniqueness}
The proof of uniqueness in the RFWRM for $d\le 2$ consists of an adaptation and modification of the Aizenman--Wehr method in infinite volume~\cite{aizenman1990rounding}. Before outlining this method, let us fix some notation. Define the \textit{order parameter generating functions} $G_{\Lambda,\pm}$ for some $\Lambda\Subset \Z^d$ and $\zeta,\tilde \zeta\in \W$ as
\begin{align*}
    G_{\Lambda, \pm}(\zeta,\tilde \zeta)=\log\Big(\int \exp\Big(\sum_{x\in \Lambda}\zeta_x\omega_x\Big)\, \mu_{ \pm}^{\mathbf{0}_{\Lambda}\tilde \zeta_{\Lambda^\compl}}(\diff \omega)\Big).
\end{align*}
Note that this is well-defined and measurable by Proposition~\ref{Prop:Existence}
and that we, from here onwards,  drop the dependence on $\lambda$ in $\mu^{\zeta, \lambda}_{\pm }$ to ease notation. We denote the difference between maximal and minimal boundary conditions, expected over the random fields outside of the finite volume, by
\begin{align}\label{Eq:DefFLambda}
    F_\Lambda(\zeta)=\E[G_{\Lambda, +}(\zeta,\eta)-G_{\Lambda, -}(\zeta, \eta)].
\end{align}
Note that the functions $F_\Lambda$ manifestly draw on infinite-volume quantities as opposed to the random variables $F_{\Lambda,\varepsilon}$ defined in Section~\ref{sec_nonuniqueness} in~\eqref{eq:FLambda}. 
Furthermore, define the \textit{outer boundary} of $A$ by $\partial^\star A=\{j\in A^\compl \setslash \exists i\in A\colon j\sim i\}$ and let $\tau_x$ denote the \textit{difference of the local magnetization} at $x$ taken between the quenched plus- and minus Gibbs measure, i.e., for $\zeta\in \W$ 
\begin{align}\label{Eq:Deftau}
\tau_x(\zeta)=\mu_{+}^\zeta(\sigma_x)-\mu_{-}^\zeta(\sigma_x).
\end{align}
Here, we used the notation $\sigma(\omega)=\omega$ and recall that $\nu(f)=\int_\Omega f(\omega)\, \nu(\diff \omega)$. The core of the Aizenman--Wehr method is captured by the following two propositions. First, we prove a deterministic upper bound on $F_\Lambda$ of boundary order.
\begin{prop}[Surface order]\label{Prop:Upperbound}
   For all $\zeta\in \W$ and $\Lambda\Subset \Z^d$,
    \begin{align*}
        |F_\Lambda(\zeta)|\le \left(\log(3)+ |\lambda|+ \E[|\eta_o|]\right)|\partial^\star \Lambda|.
    \end{align*}
\end{prop}
Second, we establish a lower bound on the fluctuations of $F_\Lambda$ under the assumption that the expectation of $\tau_x(\eta)$ is larger than zero.
\begin{prop}[Fluctuations]\label{Prop:Lowerbound}
Let $d\le 2$ and $\Lambda_n=\{-n,\ldots, n\}^d$. If $\E[\tau_x(\eta)]>0$ for some $x\in \Z^d$, there is $c >0$ such that for all $t>0$
    \begin{align*}
        \liminf_{n\to \infty }\, \E\Big[\exp\Big( t F_{\Lambda_n}(\eta)/\sqrt{|\Lambda_n|}\Big)\Big]\geq \exp( c t^2).
    \end{align*}
\end{prop}
Before we prove both propositions, let us show how these upper and lower bounds readily imply our uniqueness result.
\begin{proof}[Proof of Theorem~\ref{Thm:Uniqueness}]
    Assuming that $\E[\tau_x(\eta)]>0$ for some $x$ leads, together with Propositions~\ref{Prop:Upperbound}~and~\ref{Prop:Lowerbound}, to the contradiction that for some constants $c,C>0$ and for all $t>0$
    \begin{align*}
        \exp( ct^2  )\le  \liminf_{n\to \infty }\, \E\Big[\exp\Big( t F_{\Lambda_n}(\eta)/\sqrt{|\Lambda_n|}\Big)\Big]\le\liminf_{n\to \infty} \, \exp(tC|\partial^\star \Lambda_n|/\sqrt{|\Lambda_n|} )\le \exp(2 Ct),
    \end{align*}
    where we used $d\le 2$ for the last inequality. So, introducing the infinite-volume measures $Q_\pm(\diff \omega)\coloneq \E[\mu_{\pm}^\eta(\diff \omega)]$ we have that
    \begin{align*}
    0&=\E[\tau_x(\eta)]=Q_+(\sigma_x=1)-Q_+(\sigma_x=-1)-Q_-(\sigma_x=1)+Q_-(\sigma_x=-1)=2Q_+(\sigma_x=1)-2Q_-(\sigma_x=1),
    \end{align*}
    for all $x\in\Z^d $, where we used symmetry of the distribution of $\eta$ and $\mu_{\pm}^{\zeta}$ for the last equality. By Proposition~\ref{Prop:Existence}, for almost all $\eta$ 
    \begin{align*}
        \mu_{+}^\eta(\sigma_x=1)\ge\mu_{-}^\eta(\sigma_x=1)\quad \text{and hence}\quad \mu_{+}^\eta(\sigma_x=1)=\mu_{-}^\eta(\sigma_x=1).
    \end{align*}
    This implies, again by Proposition~\ref{Prop:Existence} that $\mu_{+}^\eta(\sigma_x=0)\ge\mu_{+}^\eta(\sigma_x=0)$.
    Since by symmetry $Q_+(\sigma_x=0)=Q_-(\sigma_x=0)$, we get $\mu_{+}^\eta(\sigma_x=0)=\mu_{+}^\eta(\sigma_x=0)$. 
    Overall, we have proved that $\sigma_x$ has the same distribution under $\mu_{+}^\eta$ and $\mu_{-}^\eta$ for almost all $\eta$. By \cite[Proposition~4.12]{georgii2001random}, together with Proposition~\ref{Prop:Existence}, this is upgraded to $\mu_{+}^\eta=\mu_{-}^\eta$ for almost all $\eta$. Again by Proposition~\ref{Prop:Existence}, we arrive at
    \begin{align*}
        \mu_{+}^\eta(\sigma\in A)= \mu_{-}^\eta(\sigma\in A)\le \nu^\eta(A) \le \mu_{+}^\eta(\sigma\in A),
    \end{align*}
    almost surely for all local, increasing and measurable sets $A\subset\Omega$,
    where $\nu^\zeta$ is an arbitrary infinite-volume Gibbs measure w.r.t.\ the external field $\zeta\in \W$. Since local and increasing sets are measure determining, we proved uniqueness of Gibbs measures for almost all $\eta$.
\end{proof}
So, it is only left to prove Propositions~\ref{Prop:Upperbound} and~\ref{Prop:Lowerbound}. While for the latter we can import a result from the original work by Aizenman and Wehr, the proof of Proposition~\ref{Prop:Upperbound} relies on novel, model specific arguments to deal with the hard-core interaction. 
\begin{proof}[Proof~of~Proposition~\ref{Prop:Upperbound}]
    By symmetry of the quenched WRM, 
    \begin{align*}
        G_{\Lambda,+}(\zeta, \tilde \zeta)=G_{\Lambda,-}(-\zeta,-\tilde \zeta)
    \end{align*}
     and hence by symmetry of the random field,
    \begin{align*}
    F_\Lambda(\zeta)=\E[G_{\Lambda, -}(-\zeta,\eta)-G_{\Lambda, -}(\zeta,\eta)].
    \end{align*}
   Writing $\bar \Lambda= \Lambda\cup\partial ^\star \Lambda$, by the DLR equations
    \begin{align*}
        G_{\Lambda, -}(-\zeta,\eta)=\log\Big(\int \int \exp\Big(-\sum_{x\in \Lambda}\zeta_x\omega_x\Big)\, \mu_{\bar \Lambda, \tilde \omega}^{\mathbf{0}_\Lambda\eta_{\Lambda^\compl}}(\diff  \omega)\, \mu_{ -}^{\mathbf{0}_\Lambda\eta_{\Lambda^\compl}}(\diff \tilde \omega)\Big).
    \end{align*}
    By definition of the finite-volume WRM and by ignoring the hard-core condition between $\Lambda$ and $\partial^\star \Lambda$ to obtain an upper bound,
    \begin{align*}
Z^{\mathbf{0}_\Lambda\eta_{\Lambda^\compl}}_{\bar \Lambda,\tilde \omega }&\int \exp\Big(-\sum_{x\in \Lambda}\zeta_x\omega_x\Big)\, \mu_{\bar \Lambda, \tilde \omega}^{\mathbf{0}_\Lambda\eta_{\Lambda^\compl}}(\diff  \omega)=\sum_{\omega_{\bar \Lambda}\in \{-1,0,1\}^{\bar \Lambda }} \exp\Big(-\sum_{k\in \Lambda}\zeta_k\omega_k-H^{\mathbf{0}_\Lambda\eta_{\Lambda^\compl} }_{\bar \Lambda}(\omega_{\bar \Lambda}\tilde \omega _{\bar \Lambda^\compl})\Big)\\
        &\le \sum_{\omega_{\bar\Lambda}\in \{-1,0,1\}^{\bar \Lambda }}\exp\Big(-\sum_{x\in \Lambda}\zeta_x\omega_x-H^{\mathbf{0}}_{\Lambda}(\omega_\Lambda\mathbf{0}_{\Lambda^\compl})- H_{\partial ^\star \Lambda}^{\eta}(\omega_{\partial ^\star \Lambda}\mathbf{0}_{(\partial^\star\Lambda)^\compl} )\Big)\\
        &=\sum_{\omega_\Lambda\in \{-1,0,1\}^{ \Lambda }}\exp\Big(-\sum_{x\in \Lambda}\zeta_x\omega_x-H^{\mathbf{0}}_{\Lambda}(\omega_\Lambda\mathbf{0}_{\Lambda^\compl}) \Big)\sum_{\omega_{\partial^\star   \Lambda}\in \{-1,0,1\}^{\partial^\star   \Lambda }}\exp\big(-H_{\partial ^\star \Lambda}^{\eta}(\omega_{\partial ^\star \Lambda}\mathbf{0}_{(\partial^\star\Lambda)^\compl} )\big).
    \end{align*}
    Using again the definition of the finite-volume WRM for the first sum on the right-hand side,
    \begin{align*}
       \sum_{\omega_\Lambda\in \{-1,0,1\}^{ \Lambda }}\exp\Big(-\sum_{x\in \Lambda}\zeta_x\omega_x-H^{\mathbf{0}}_{\Lambda}(\omega_\Lambda\mathbf{0}_{\Lambda^\compl}) \Big)&=\sum_{\omega_\Lambda\in \{-1,0,1\}^{ \Lambda }}\exp\Big(\sum_{x\in \Lambda}\zeta_x\omega_x-H^{\mathbf{0}}_{\Lambda}(\omega_\Lambda\mathbf{0}_{\Lambda^\compl} ) \Big)\\
       &=\sum_{\omega_\Lambda\in \{-1,0,1\}^{ \Lambda }}\exp\Big(\sum_{x\in \Lambda}\zeta_x\omega_x-H^{\mathbf{0}_\Lambda\eta_{\Lambda^\compl} }_{\bar\Lambda}(\omega_\Lambda \mathbf{0}_{\partial^\star \Lambda}\tilde \omega_{\bar \Lambda^\compl } ) \Big)\\
        &\le \sum_{\omega_{\bar\Lambda}\in \{-1,0,1\}^{ \bar \Lambda }}\exp\Big(\sum_{x\in \Lambda}\zeta_x\omega_x-H^{\mathbf{0}_\Lambda\eta_{\Lambda^\compl} }_{\bar \Lambda}(\omega_{\bar\Lambda}\tilde \omega_{\bar \Lambda^\compl } ) \Big)\\
    &=Z^{\mathbf{0}_\Lambda\eta_{\Lambda^\compl}}_{\bar \Lambda,\tilde \omega }\int \exp\Big(\sum_{x\in \Lambda}\zeta_x\omega_x\Big)\, \mu_{\bar \Lambda, \tilde \omega}^{\mathbf{0}_\Lambda\eta_{\Lambda^\compl} }(\diff  \omega),
    \end{align*}
    where extending the summation from $\omega_\Lambda$ to $\omega_{\bar \Lambda}$ in the third line only leads to an upper bound, which holds for all $\tilde \omega$.
    Furthermore, the second sum can be estimated by
    \begin{align*}
       \sum_{\omega_{\partial^\star   \Lambda}\in \{-1,0,1\}^{\partial^\star   \Lambda }}\exp\big(-H_{\partial ^\star \Lambda}^{\eta}(\omega_{\partial ^\star \Lambda}\mathbf{0}_{(\partial^\star\Lambda)^\compl} )\big)&\le \sum_{\omega_{\partial^\star   \Lambda}\in \{-1,0,1\}^{\partial^\star   \Lambda }}\exp\Big(|\lambda|\, |\partial ^\star \Lambda| +\sum_{x\in \partial^ \star \Lambda} |\eta _x|\Big)\\
       &\le  3^{|\partial ^\star \Lambda|} \exp\Big(|\lambda|\, |\partial ^\star \Lambda| +\sum_{x\in \partial^ \star \Lambda} |\eta _x|\Big).
    \end{align*}
    Putting things together, we have that 
    \begin{align*}
        G_{\Lambda, -}(-\zeta,\eta)\le G_{\Lambda, -}(\zeta,\eta)+\log\Big(3^{|\partial ^\star \Lambda|} \exp\Big(|\lambda|\, |\partial ^\star \Lambda| +\varepsilon\sum_{k\in \partial^ \star \Lambda} |\eta _k|\Big)\Big)
    \end{align*}
    and  
    \begin{align*}
        F_\Lambda(\zeta)\le \E\Big[\log\Big(3^{|\partial ^\star \Lambda|} \exp\Big(|\lambda|\, |\partial ^\star \Lambda| +\varepsilon\sum_{x\in \partial^ \star \Lambda} |\eta _x|\Big)\Big)\Big]\le |\partial^\star \Lambda | \left(\log(3)+ |\lambda|+ \E[|\eta_o|]\right).
    \end{align*}
    The claim follows since $F_\Lambda(-\zeta)=-F_\Lambda(\zeta)$.
\end{proof}
As mentioned earlier, the proof of Proposition~\ref{Prop:Lowerbound} can be deduced from the more general result~\cite[Proposition~6.1]{aizenman1990rounding}. For a good overview on the assumptions used there, let us formulate it here in our notation. 
\begin{prop}[{\cite[Proposition~6.1]{aizenman1990rounding}}]\label{Prop:Properties_translationcovariant}Let $(\eta_x)_{x\in \Z^d}$ be an i.i.d.~random field satisfying \mbox{\ref{Ass:etaSymmetry} and~\ref{Ass:etaNontrivial}} from Section~\ref{sec_main}. Assume that for every $x,y\in \Z^d$, $\Lambda\Subset\Z^d$ and $\zeta\in \W$ the following holds.
\begin{enumerate}[label=(\roman*)]
        \item $\tau$ is translation covariant, that is
        \begin{align*}
            \tau_{x+y}(\zeta)=\tau_x(T_y\zeta)\quad\text{ where }T_y\zeta=(\zeta_{\tilde x-y})_{\tilde x\in \Z^d}.
        \end{align*}
        \item If $x\in \Lambda$, we have $ \partial_{x}F_{\Lambda}(\zeta)=\E[\tau_x(\zeta_{\Lambda}
\eta_{\Lambda^\compl})]$, where $\partial_x$ denotes the derivative w.r.t.\ the component~$\zeta_x$,
        \item $\tau\ge 0$, $|\tau_o(\zeta)|\le 2$ and $|\partial_o\tau_o(\zeta)|\le 2$, and 
        \item $\E[F_\Lambda(\eta)]=0$.
    \end{enumerate}
    Then, there is a function $\gamma\colon [0,\infty )\rightarrow [0,\infty )$ such that $\gamma(r)>0$ for $r>0$ and such that for all $t>0$
    \begin{align*}
        \liminf_{n\to \infty }\, \E\Big[\exp\Big( t F_{\Lambda_n}(\eta)/\sqrt{|\Lambda_n|}\Big)\Big]\geq \exp\big( t^2  \gamma\big(\E[\tau_o(\eta)]\big)^2 /2\big).
    \end{align*}
\end{prop}
This result is in fact a special case of~\cite[Proposition~6.1]{aizenman1990rounding} in the sense that they allow any functions $F_\Lambda$ and $\tau$ satisfying the above assumptions and not just our particular choice from~\eqref{Eq:DefFLambda} and~\eqref{Eq:Deftau}. Note that the positivity of $\gamma$ is not explicitly stated in~\cite[Proposition~6.1]{aizenman1990rounding} but only later in~\cite[Proposition~A.3.2.~iii, Page~524]{aizenman1990rounding}. Note that in~\cite[Proposition~6.1]{aizenman1990rounding} it was stated that this positivity criterion on $\gamma$ is applicable if $\tau$ is monotone but their proof shows applicability if $\tau\ge 0$, which holds true by the monotonicity of the RFWRM. Note also that $\E[\exp(t\eta_o)]<\infty$ for all $t>0$ is assumed in~\cite[Equation~(2.1)]{aizenman1990rounding}. However, an inspection of their proof for the lower bound on the exponential moment generating 
function in~\cite[Proposition~A.2.1]{aizenman1990rounding} shows that it holds also when the exponential moment generating function is infinite. 
Indeed, this requires the use of the notion of a non-negative submartingale that 
may not be integrable, along with a version of~\cite[Lemma~A.2.2]{aizenman1990rounding} for conditional probabilities. The extension to possibly non-integrable submartingales poses no problems here, as it is 
only used to derive monotonicity of expectations. Therefore, we may drop the assumption $\E[\exp(t\eta_o)]<\infty$ on the single-site random-field variables in Proposition~\ref{Prop:Properties_translationcovariant}.
With this powerful tool (whose proof relies on martingale methods) at hand, we are ready to prove the lower bound. 
\begin{proof}[Proof~of~Proposition~\ref{Prop:Lowerbound}]
We check the assumptions of Proposition~\ref{Prop:Properties_translationcovariant}.
    \begin{enumerate}[label=$(\roman*)$]
        \item Let $\Lambda_n\uparrow \Z^d$ and $y\in \Z^d$ and note that
        \begin{align*}
        \mu_{\Lambda_n,+}^{\zeta}(\sigma_{x+y})=\mu^{T_y\zeta}_{\Lambda_n-y,+}(\sigma_{x}).
        \end{align*}
        Since also $(\Lambda_n-y)\uparrow \Z^d$, the claim follows when letting $n$ tend to infinity by Proposition~\ref{Prop:Existence}.
        \item A key insight of the Aizenman--Wehr method is the finite perturbation property
        \begin{align*}
            \mu_{\pm}^{\tilde \zeta+\zeta_\Lambda }(f(\sigma))=\mu_{\pm}^{\tilde \zeta}\Big(f(\sigma)\exp\Big(\sum_{x\in \Lambda }\zeta_x \sigma_x\Big)\Big)/\mu_{\pm}^{\tilde \zeta}\Big(\exp\Big(\sum_{x\in \Lambda }\zeta_x \sigma_x)\Big),
        \end{align*}
        for all local test functions $f$.
        One can check that this is true when replacing $\smash{ \mu_{\pm}^{\tilde \zeta+\zeta_\Lambda }}$ with $\smash{ \mu_{\Lambda_n,\pm}^{\tilde \zeta+\zeta_\Lambda }}$ where $\Lambda_n\Subset \Z^d$. In the above form, it follows by choosing $\Lambda_n\uparrow \Z^d$ and applying Proposition~\ref{Prop:Existence}. 
        Hence, 
        \begin{align*}
            \partial_x G_{\Lambda, \pm}(\zeta,\tilde \zeta)=\mu_{\pm}^{\mathbf{0}_\Lambda\tilde \zeta_{\Lambda^\compl}}\Big(\sigma_x\exp\Big(\sum_{y\in \Lambda }\zeta_y \sigma_y\Big)\Big)/\mu_{\pm}^{\mathbf{0}_\Lambda\tilde \zeta_{\Lambda^\compl}}\Big(\exp\Big(\sum_{y\in \Lambda }\zeta_y \sigma_y)\Big)=\mu^{\zeta_\Lambda\tilde \zeta_{\Lambda^\compl}}_{\pm}(\sigma_x)
        \end{align*}
        and
        \begin{align*}
            \partial_{\zeta_x}F_{\Lambda}(\zeta)
            =\E[\partial_{\zeta_x} G_{\Lambda, +}(\zeta,\eta)-\partial_{\zeta_x} G_{\Lambda, -}(\zeta,\eta)]=\E[\tau_x(\zeta_{\Lambda}\eta_{\Lambda^\compl})].
         \end{align*}
        \item The first inequality follows by Proposition~\ref{Prop:Existence}, the second by definition and the third is implied by 
        \begin{align*}
            \partial_o\tau_o(\zeta)=\mu_{+}^\zeta(\sigma_o^2)-\mu_{+}^\zeta(\sigma_o)^2-\mu_{-}^\zeta(\sigma_o^2)+\mu_{-}^\zeta(\sigma_o)^2,
        \end{align*}
        where we again used the finite-perturbation property.
        \item Since $G_{\Lambda,+}(\zeta, \tilde \zeta)=G_{\Lambda,-}(-\zeta,-\tilde \zeta)$, it holds that
        \begin{align*}
            \E[F_\Lambda(\eta)]=\E[ G_{\Lambda, +}(\eta,\eta)-G_{\Lambda, -}(\eta, \eta)]=0.
        \end{align*}
    \end{enumerate}
    This finishes the proof. 
\end{proof}

\section*{Acknowledgement} 
BJ received support by the Leibniz Association within the Leibniz Junior Research Group on \textit{Probabilistic Methods for Dynamic Communication Networks} as part of the Leibniz Competition (grant no.\ J105/2020). The research of BJ and DK is funded by Deutsche Forschungsgemeinschaft (DFG) through the SPP2265 within the Project P27.

\appendix\section{Appendix}
\subsection{Proof of Lemma~\ref{Lemma:CountingSimpleSets}}\label{Sec:Appendix_Boundary}
This section is devoted to the proof of Lemma~\ref{Lemma:CountingSimpleSets} and some more results on connectivity. We start with the proof of Item~$(i)$ in Lemma~\ref{Lemma:CountingSimpleSets}.
\begin{lemma}\label{Lemma:EmptyBoundary}
    For finite $ \Lambda\in \mathfrak{A}_2$  and $\omega \in \Omega^+_\Lambda$ with $\omega_o\neq 1$, it holds that $A_\omega\in \mathfrak{A}_2$ and $\omega_{\partial ^1A_\omega} =\mathbf{0}$. 
\end{lemma}
\begin{proof}
    First, we prove that $A_\omega$ is $2$-connected. If $x,y\in  A_\omega$, by definition there is a $2$-connected path from $x$ to some $\tilde x\in B_\omega$ and a $2$-connected path from $y$ to some $\tilde y \in B_\omega$. Since $B_\omega$ is by definition $2$-connected, $A_\omega$ is $2$-connected. Also $A_\omega^\compl$ is connected since by definition there exists for every $x\in A_\omega^\compl $ a path from $x$ to $\Lambda^\compl$ which is $2$-connected. 
    Furthermore, for $x\in\partial^1 A_\omega$ we note that $x\in B_\omega$ and hence $\omega_x\neq 1$. There is a $y\in A_\omega^\compl\subset B_\omega ^\compl$ with $y\, \smash{\overset{1}{\sim}}\, x$. By maximality of $B_\omega$, we have $\omega_y=1$ and since $\omega\in \Omega^+_\Lambda$ we know that $\omega _x\neq -1$. Overall, $\omega_x=0$ and $\omega_{\partial ^1A_\omega} =\mathbf{0}$. 
\end{proof}
A key ingredient for the proof of Item~$(ii)$ in Lemma~\ref{Lemma:CountingSimpleSets} is that $\partial^2A$ is $2$-connected. While this result seems intuitive when considering simple examples in $\Z^2$, it is not so easy to prove it in full generality. In the following proof, we use the relatively recent approach to this problem as presented in \cite{Ti13}.

\begin{lemma}\label{Lemma:ConnectivityBoundary}
  Let $d \ge 2$. For $m=1,2$ and $A\in \mathfrak{A}_m$, $\partial^{m} A\subset \Z^d $ is $d$-connected.
\end{lemma}
\begin{proof}
    For $m=1$, the claim follows by~\cite[Proposition~B.82]{FrVe17}. Note that this result even gives that $\partial^1 A$ is $\star$-connected, where two vertices $x$ and $y$ are $\star$-neighboring if $\|x-y\|_\infty =1$, which is stronger than \mbox{$d$-connectedness}. For $m=2$, the proof consists in checking the requirements of~\cite[Theorem~3]{Ti13}. First, define the graph $G=(V,E_2)$ with vertices $V=\Z^d$ and edges defined by $E_m=\{\{x,y\}\subset V\setslash x\overset{m}{\sim}y\}$ for $m\in \N$. Note that in the aforementioned theorem, they need also a second graph $G^+=(V,E_d)$. Some $\tilde E\subset E_m$ is called {\em $m$-cycle} if for every $x\in \Z^d$, the number of $\{x,y\}\in \tilde E$ is even. Furthermore, a cycle is {\em $m$-chordal} if every two vertices with edges in this cycle are $m$-adjacent. To apply~\cite[Theorem~3]{Ti13}, we need to find a $2$-generating set of $d$-chordal cycles $\Delta $, which means that every $2$-cycle $\tilde E$ can be represented by 
    \begin{align*}
        \tilde E=T_1\ominus \cdots \ominus T_n
    \end{align*}
    for some $T_1,\ldots, T_n\in \Delta$, where we used that the symmetric difference $\ominus$ is associative to make sense of the above notation. In our case, we use $\Delta=\Delta_1\cup\Delta_2$, where $[x_1,x_2,x_3]=\{\{x_1,x_2\},\{x_2,x_3\},\{x_3,x_1\}\}$ are triangles and
    \begin{align*}
        \Delta_1&=\{[x_1,x_2,x_3]\setslash \forall y_1,y_2\in\{x_1,x_2, x_3\}\colon \|y_1-y_2\|_\infty =1\}\\
        \Delta_2&=\{[x_1,x_2,x_3]\setslash x_1\overset{1}{\sim }x_2,\,x_2\overset{1}{\sim}x_3  \}.
    \end{align*}
    Indeed, every $T\in \Delta$ is $d$-chordal. Next, we prove that $\Delta$ is $2$-generating. Every $e=\{x,y\}\in E_2\setminus E_1$ of $2$-adjacent vertices can be reduced to a $1$-adjacent edge situation via the symmetric difference operations. Indeed, since we can find $z$ such that $x\,\smash{\overset{1}{\sim}}\, z$ and $y\, \smash{\overset{1}{\sim}}\, z$, we have that $[x,z,y]\in \Delta_2$ and
    \begin{align*}
        \{e\}\ominus [x,z,y]\subset E_1.
    \end{align*}
    By iterating this argument, we can find for every cycle $C\subset E_2$ some $T_1,\ldots, T_n\in \Delta_2$ such that 
    \begin{align*}
        \tilde C\coloneq C\ominus T_1\ominus \ldots \ominus T_n
    \end{align*}
    is a $1$-cycle. By~\cite[Lemma~B.81]{FrVe17}, there is a representation of $\tilde C$ in terms of $\Delta_1$, i.e., there are $T_{n+1},\ldots, T_m\in \Delta_1$ such that
    \begin{align*}
        C=\tilde C\ominus T_1\ominus \ldots \ominus T_n=T_1\ominus\ldots\ominus T_m
    \end{align*}
    and $\Delta$ is indeed $2$-generating. As a result, \cite[Theorem~3]{Ti13} states that the set $\partial ^2 A$ is $d$-connected for $A\in \mathfrak{A}_2$. To be precise, \cite[Theorem~3]{Ti13} proves that $\smash{\delta ^G_{\text{vis}_G(\infty)}A}$ is $d$-connected, which is defined there as the part of the boundary of $A$ that is connected to infinity via a $2$-connected path that stays in $A^\compl$. In our case however,  $\smash{\partial^2 A=\delta ^G_{\text{vis}_G(\infty)}A}$ since $A\in \mathfrak{A}_2$, which yields the claim.
\end{proof}
As a result, we can count the number of connected sets of size $n$ instead of $2$-simply connected sets with boundary of size $n$. This can be done by a depth-first search. We include a concrete construction of the search algorithm as we need it also later in the coarse-graining argument in Appendix~\ref{Sec:AppCoarseGraining}.
\begin{lemma}\label{Lemma:CountingConnectedSets}
For every $m\in \N$, there is a $C>0$ such that for all $n\in \N$ and $x\in \Z^d $ 
\begin{align*}
    |\{\gamma\subset \Z^d\setslash \gamma\text{ is $m$-connected, }|\gamma|=n\text{ and }x\in \gamma \}|\le C^n.
\end{align*}
\end{lemma}
\begin{proof}
    Let $\gamma\subset \Z^d$, which is $m$-connected, $|\gamma|=n$ and $x\in \gamma $. Consider $ \gamma $ as a graph, i.e., the vertices are $V=\gamma$ and the edges on $V$ are $E=\{\{x,y\}\subset  V\setslash x\overset{m}{\sim } y\}$. Since $ \gamma$ is $m$-connected also $(V,E)$ is connected and we can find a spanning tree $(V,\tilde E)$ of $(V,E)$ with $|\tilde E|=n-1$. Running depth-first search on $(V,\tilde E)$ can be interpreted as an $m$-connected path $\pi_\gamma\colon\{0,\ldots , 2(n-1)\}\rightarrow V$. More precisely, we start in $\pi_\gamma(0)=x\in \gamma$ and then follow edges, i.e., $\{\pi_\gamma (t),\pi_\gamma(t+1)\}\in \tilde E$ until we hit a leaf in $\tilde E$. Then, we go back from the leaf exactly in the opposite direction by only using edges that were followed exactly once until we find a vertex with an edge that we have not followed yet. Starting from there, we follow yet unfollowed edges until we find a leaf and go back again. This is done until we do not find an unfollowed edge while backtracking and we reach the starting vertex $\pi_\gamma(0)$. By this procedure, we follow every edge in $\tilde E$ exactly twice which is why $\pi_\gamma(0)=\pi_\gamma(2(n-1))$. By definition, $\pi_\gamma$ is an $m$-connected path. Furthermore, the mapping $\gamma\mapsto \pi _\gamma$ is injective. So,
    \begin{align*}
     |\{\gamma\subset \Z^d\setslash \gamma\text{ is $m$-connected, }|\gamma|=n, \, x\in \gamma \}|&\le 
         |\{\pi \colon \{0,\ldots, 2(n-1)\}\rightarrow V\setslash \pi \text{ is $m$-connected, }\pi(0)=x\}|\\
        &\le |\{y\in \Z^d \setslash \|y\|_1\le m\}|^{2(n-1)},
    \end{align*}
    which proves the claim. 
\end{proof}
We are ready to prove the Item~$(ii)$ in Lemma~\ref{Lemma:CountingSimpleSets}, which together with Lemma~\ref{Lemma:EmptyBoundary}, proves the entire lemma.
\begin{proof}[Proof of Lemma~\ref{Lemma:CountingSimpleSets}]
    Let $A\in \mathfrak{A}_2$ with $|\partial^2  A|=n$. From an intuitive standpoint it is clear that, in $d\ge 2$, for all $x\in \partial ^2 A$ we have $\|x\|_1\le n$. One way to see rigorously that there is at least one $x\in \partial^2 A$ with $\|x\|_1\le n$ is by starting from $o\in A$ and moving along the first unit direction until we find an $x\in \partial ^2A$. For every point we cross until we find this $x$, we can find, in $d \ge 2$, some pairwise different elements in $ \partial ^2A$ by moving into the second unit direction. Since there can be at most $n$ such points in $\partial ^2 A$, it holds that $\|x\|_1\le n$. This implies that 
    \begin{align*}
        |\{A\in \mathfrak{A}_2 \setslash  |  \partial ^2 A|=n \}|\le \sum_{x\colon \|x\|_1\le n}|\{A\in \mathfrak{A}_2\setslash |\partial^2 A|=n\text{ and }x\in \partial ^2 A \}|.
    \end{align*}
    Next, we prove that $A\mapsto \partial ^2 A$ is injective on $\{\mathfrak{A}_2\setslash |\partial ^2 A|=n\}$. Assume that $A,B\in \mathfrak{A}_2$ satisfy $\partial ^2A=\partial ^2 B$ and take $x\in  A$ but $x\notin B$. Since $B^\compl$ is $2$-connected there is a $2$-connected path with $\pi(t)\notin B$ for all $t$ such that $\pi$ starts in $x$ and ends in the complement of $\{-n,\ldots, n\}^d$. Therefore, there is a minimal $t^\star$ with $\pi(t^\star)\notin  A$. Thus, $ \pi(t^\star-1)\in \partial^2 A =\partial ^2 B$, which is a contradiction to $\pi (t^\star -1)\notin B$. Furthermore, by Lemma~\ref{Lemma:ConnectivityBoundary} we have that $\partial^2 A$ is $d$-connected. As a result, 
     \begin{align*}
        |\{A\in \mathfrak{A}_2\setslash |  \partial ^2 A|=n \}|\le \sum_{x\colon \|x\|_1\le n}|\{\gamma \subset \Z^d\setslash \gamma\text{ is $d$-connected, }|\gamma|=n\text{ and }x\in \gamma \}|
    \end{align*}
    and the claim follows by  Lemma~\ref{Lemma:CountingConnectedSets}.
\end{proof}
As noted earlier, Lemma~\ref{Lemma:CountingConnectedSets} is also useful for the coarse-graining argument in Appendix~\ref{Sec:AppCoarseGraining} due to the following. 
\begin{lemma}\label{Lemma:permutationBoundary}
    There is $C>0$ such that, for every $m$-connected $\gamma\subset \Z^d$ and all $x_1,\ldots,x_k\in \gamma$ pairwise different, there is a permutation $p\colon \{1,\ldots, k\}\rightarrow \{1,\ldots,k\}$ with
\begin{align*}
    \sum_{i=2}^k\|x_{p(i)}-x_{p(i-1)}\|_1\le C| \gamma|.
\end{align*}
\end{lemma}
\begin{proof}
    Define $\pi_\gamma$ as in the proof of Lemma~\ref{Lemma:CountingConnectedSets}. Since $\pi_\gamma$ maps surjective onto $\gamma$, we can find $p$ such that there are increasing $n_1<\ldots <n_k$ with $x_{p(i)}=\pi_\gamma(n_i)$. By using that $\pi_\gamma$ is $m$-connected by construction, there is a constant $C>0$ such that
    \begin{align*}
    \sum_{i=2}^k\|x_{p(i)}-x_{p(i-1)}\|_1=\sum_{i=2}^k\|\pi(n_i)-\pi(n_{i-1})\|_1\le\sum_{i=1}^{2(|\gamma|-1)}\|\pi(i)-\pi(i-1)\|_1\le C|\gamma |.
    \end{align*}
    This finishes the proof. 
\end{proof}
Finally, we argue that different definitions of the boundary lead to equivalent sizes of the boundary, which we used in the Peierls argument in Section~\ref{sec_nonuniqueness}. In Appendix~\ref{Sec:AppCoarseGraining}, we will need a similar statement also for the \textit{edge boundary}, i.e.,
 \begin{align*}
        \partial^\prime A=\{(x,y)\in A\times A^\compl\setslash x\overset{1}{\sim}y\},
    \end{align*}
which is why we include it in the following lemma. 
\begin{lemma}\label{Lemma:Equivalence_Boundaries}
    There are $C_1,C_2>0$ such that for all $A\in \mathfrak{A}_2$ it holds that
    \begin{align*}
        |\partial^1 A|\le |\partial ^2 A|\le C_1 |\partial ^\prime A|\le C_2|\partial ^1 A|.
    \end{align*}
\end{lemma}
\begin{proof}
    The first inequality follows by definition. Towards the second inequality, note that there is for every $x\in \partial^2 A$ an $e(x)=(e_1(x),e_2(x))\in \partial ^\prime A$ with $e_1(x)\, \smash{\overset{1}{\sim }}\,x$ or $e_1(x)=x$. Since every point has at most $2^{d}$ neighbors w.r.t.\ $\smash{\overset{1}{\sim}}$, the second inequality holds with $C_1=(2^d+1)$. The third inequality follows since for every $e=(x_e,y_e)\in \partial^\prime A$ also $x_e\in \partial ^1 A$ and there are at most $2^d$ edges $e$ with the same $x_e$.
\end{proof}
\subsection{Proof of Proposition~\ref{Prop:CoarseGraining}}\label{Sec:AppCoarseGraining}
In this section, we prove Proposition~\ref{Prop:CoarseGraining} by revisiting the coarse-graining arguments in~\cite{FiFrSp84} in the more general framework with $m=1,2$ and $d\ge 3$. Note that we formulated the claim by using the $1$-boundary, i.e., $|\partial^1 A|$ but by Lemma~\ref{Lemma:Equivalence_Boundaries} we can switch without loss of generality to $|\partial^\prime A|$. Now, a first idea could be to expand the probability of interest na\"{\i}vely by
\begin{align*}
    P\coloneq \P(\exists A\in \mathfrak{A}_m\colon |F_{A,\varepsilon} |\ge z |\partial^\prime A|)\le \sum_{ A\in \mathfrak{A}_m}\P(|F_{A,\varepsilon} |\ge z |\partial^\prime A|),
\end{align*}
but this estimate turns out to be too crude. Instead, \cite{FiFrSp84} defines
\begin{align} \label{Eq:ChoiceN}
    N(n)&=\lceil  \log_2(n^{1/d})\rceil
\end{align}
as well as certain coarse-graining mappings $A\mapsto A_k$ for every $k\in\Z_{\ge 0} $ with $A_0=A$ and $A_k\Subset \Z^d$, which will be defined properly later on. Then, one can expand
\begin{align*}
    P&\le\sum_{n=1}^\infty  \P(\exists A\in \mathfrak{A}_m, \, |\partial^\prime A|=n\colon |F_{A,\varepsilon} |\ge z n) \\
    &\le\sum_{n=1}^\infty  \P\Big(\exists A\in \mathfrak{A}_m, \, |\partial^\prime A|=n\colon \sum_{k=1}^{N(n)-1}|F_{A_{k},\varepsilon}-F_{A_{k-1},\varepsilon} |+|F_{A_{N(n)},\varepsilon}|\ge z n\sum_{k=1}^{N(n)}1/(2k)^2\Big),
\end{align*}
where we used that $\sum_{k=1}^\infty 1/(2k)^2=\pi^2/24\le 1$. Estimating further, we get that
\begin{align*}
    P&\le \sum_{n=1}^\infty\P\big(\exists A\in \mathfrak{A}_m,\, |\partial^\prime A|=n\colon  |F_{A_{N(n)},\varepsilon }|\ge zn/(2N(n))^2\big)\\
    &\hspace{3cm} +\sum_{n=1}^\infty \sum_{k=1}^{N(n)-1}\P\big(\exists A\in \mathfrak{A}_m,\, |\partial^\prime A|=n\colon  |F_{A_{k},\varepsilon }-F_{A_{k-1},\varepsilon }|\ge zn/(2k)^2\big) \\
    &\le\sum_{n=1}^\infty \sum_{\substack{B\colon \exists A\in \mathfrak{A}_m\\
    |\partial^\prime A|=n, B=A_{N(n)}} }\P\big(  |F_{B,\varepsilon }|\ge zn/(2N(n))^2\big)\\
    &\hspace{3cm} +\sum_{n=1}^\infty \sum_{\substack{(B_k)_k\colon \exists A\in \mathfrak{A}_m\\
    |\partial^\prime A|=n, B_k=A_{k}} }\sum_{k=1}^{N(n)-1}\P\big(  |F_{B_k,\varepsilon}-F_{B_{k-1},\varepsilon }|\ge zn/(2k)^2\big).
\end{align*}

As a result, by the Assumption~\eqref{Eq:PropertyF}, there are $c,C>0$ such that for all $z,\varepsilon >0$
\begin{equation}\label{Eq:Estimate}
\begin{split}
P&\le  C\sum_{n=1}^\infty \sum_{\substack{B\colon \exists A\in \mathfrak{A}_m\\
    |\partial^\prime A|=n, B=A_{N(n)}} }\exp\Big(-\frac{cz^2n^2}{N(n)^4\varepsilon^2 |B |}\Big) +C\sum_{n=1}^\infty \sum_{\substack{(B_k)_k\colon \exists A\in \mathfrak{A}_m\\
    |\partial^\prime A|=n, B_k=A_{k}} }\sum_{k=1}^{N(n)-1}\exp\Big(-\frac{cz^2n^2}{k^4\varepsilon^2 |B_k\ominus B_{k-1}|}\Big).
    \end{split}
\end{equation}
Certainly, the choice of $N(n)$ seems a bit arbitrary at this point and $A\mapsto A_k$ is not yet defined, but the roadmap towards the proof of Proposition~\ref{Prop:CoarseGraining} is already visible: First, find an upper bound on $|B|$ and $|B_{k}\ominus B_{k-1}|$ that is uniform over all $B_k$ that are attained as a coarse-grained version $A_k$ of $A$ and, second, find an upper bound on the number of summands, i.e.,
\begin{align}\label{Eq:DefinitionMnk}
    M_{n,k}=|\{B\colon \exists A\in \mathfrak{A}_m,  |\partial^\prime A|=n,\, B=A_{k} \}|.
\end{align}
Both bounds should be so sharp that they guarantee that the series above converges and is small for large~$z/\varepsilon$.

We start by properly defining the coarse graining $A\mapsto A_k$ that follows a majority rule. For $k\in \Z_{\ge 0}$, define the set of cubes with side length $2^k$ by
\begin{align*}
    \mathcal{C}_k=\{c_i\setslash i\in \Z^d \}, \quad \text{where }c_i=i2^k+\{0,\ldots 2^k-1\}^d.
\end{align*}
We say that $c\in \mathcal{C}_k$ is admissible w.r.t.~$A$ if 
\begin{align*}
    |c\cap A|\ge 2^{dk-1}.
\end{align*}
Finally, define $A_k$ as the union of all $c\in\mathcal{C}_k $ that are admissible w.r.t.\ $A$. It is important to note that $A_k$ need not be connected even if $A$ is connected. Towards the bounds on $|A_k|$ and $|A_{k}\ominus A_{k+1}|$, we first show that around the boundary of $A_k$ also $A$ must have a piece of boundary. 
\begin{lemma}\label{Lemma:CoarseGrainedBoundary}
Let $k\in\Z_{\ge0}$ and assume that $c_i, c_{j}\in \mathcal{C}_k$ are adjacent cubes, i.e., $i\overset{1}{\sim}j$, such that $c_i$ is admissible and $c_{j}$ is not. Then, there is constant $C>0$ independent of $c_i,c_j$ and $k$ such that for all $A\Subset \Z^d$
\begin{align*}
|\partial^\prime A \cap \{\{x,y\}\setslash x\in c_i\cup c_j\}|\ge C2^{(d-1)k}.
\end{align*}
\end{lemma}
\begin{proof}
Without loss of generality, we can restrict to $k\in \Z_{>1}$. Indeed, we always know that for $c_i$ admissible and $c_j$ not, there are $x\in c_i\cap A$ and $y\in c_j\cap A^\compl$. In particular,  $|(c_i\cup c_j)\cap \partial ^1 A|\ge 1$, which proves the claim for $k\le 1$ by choosing $C$ small enough. 

Denote $\ell=i-j$ and note that $\ell$ is a unit vector modulo some sign. Furthermore, define the mapping $\psi\colon [0,2^k]\rightarrow \R$ by setting for $n=0,\ldots, 2^k$
\begin{align*}
    \psi(n)=|A\cap (2^kj+n\ell+[0,2^k)^d)|2^{-kd}
\end{align*}
and defining $\psi$ on $[0,2^k]$ as the linear interpolation of these points. In words, we slowly shift the cube $c_j$ so that it becomes $c_i$ and meanwhile count the fraction of points in the respective cube that are in $A$. Since $c_i$ is admissible while $c_j$ is not, it holds that 
\begin{align*}
    \psi(0)\le 1/2\quad \text{ and }\quad \psi(2^k)\ge 1/2.
\end{align*}
By the intermediate value theorem there exists a $t\in [0,2^k]$ such that $\psi(t)=1/2$ and therefore there exists an $n=0,\ldots, 2^k$ such that 
\begin{align*}
    1/4\le 1/2- 2^{(d-1)k}2^{-dk}\le \psi(n)\le 1/2+ 2^{(d-1)k}2^{-dk}\le 3/4. 
\end{align*}
The isoperimetric inequality on $\Z^d$, see, e.g., \cite[Theorem~B.78]{FrVe17}, states that for all $B\Subset \Z^d$
\begin{align}\label{Eq:IsoperimetricInequality}
    |\partial ^\prime B|\ge 2d |B|^{(d-1)/d}.
\end{align}
When applying this to $B\coloneq A\cap (2^kj+n\ell+[0,2^k)^d)$ and to $\tilde B\coloneq A^\compl \cap (2^kj+n\ell+[0,2^k)^d)$, it follows that
\begin{align*}
    (|\partial^\prime  B|+|\partial^\prime  \tilde B|)2^{-k(d-1)}&\ge 2d(|B|2^{-kd})^{(d-1)/d}+2d(|\tilde B|2^{-kd})^{(d-1)/d}= 2d\psi(n)^{(d-1)/d}+2d(1-\psi(n))^{(d-1)/d}.
\end{align*}
The function $p\mapsto p^{(d-1)/d}+(1-p)^{(d-1)/d}$ is minimized on $[1/4,3/4]$ at $p=1/4$. Therefore, when choosing 
\begin{align*}
    C\coloneq (2d(1/4)^{(d-1)/d}+2d(3/4)^{(d-1)/d}-2d)/2>0,
\end{align*}
it holds that
\begin{align*}
    |\partial^\prime  B|2^{-k(d-1)}+|\partial^\prime  \tilde B|2^{-k(d-1)}\ge 2d(1/4)^{(d-1)/d}+2d(3/4)^{(d-1)/d}=2 C+2d.
\end{align*}
The claim follows by 
\begin{align*}
     |\partial^\prime  A\cap \{\{x,y\}\setslash x\in c_i\cup c_j\}|&\ge |\partial^\prime  A \cap \{\{x,y\}\setslash x\in (2^kj+n\ell+[0,2^k)^d)\cap \Z^d\}|\\
     &\ge (|\partial^\prime  B|+|\partial^\prime \tilde B|-2d 2^{k(d-1)})/2\\
     &\ge C2^{k(d-1)},
\end{align*}
     as desired.
\end{proof}
With this argument at hand, we can bound $|A_k\ominus A_{k-1}|$ and $|A_k|$ in terms of $|\partial^\prime A|$.
\begin{lemma}\label{Lemma:Bounds_on_A_k}
There is $C>0$ such that for all $k\in \Z_{\ge 0}$ and $A\Subset \Z^d$
\begin{align*}
    |\partial^\prime A_k|\le C|\partial ^\prime A|,\qquad
    | A_{k+1}\ominus A_{k}|\le C2^k|\partial ^\prime A|\qquad \text{and }\qquad 
    |A_k|\le C|\partial^\prime  A|^{d/(d-1)}.
\end{align*}
\end{lemma}
\begin{proof}
With $C$ as in Lemma~\ref{Lemma:CoarseGrainedBoundary},
\begin{align*}
     |\partial^\prime A_k|=\sum_{\substack{\{i,j\}\subset \Z^d\colon i\overset{1}{\sim}j\\ c_i\text{ admissible}, \, c_j\text{ not }}}\hspace{-0.4cm}2^{(d-1)k}\le \sum_{\substack{\{i,j\}\subset \Z^d\colon i\overset{1}{\sim}j\\ c_i\text{ admissible}, \, c_j\text{ not }}}\hspace{-0.4cm} C^{-1}|\partial^\prime A \cap \{\{x,y\}\setslash x\in c_i\cup c_j\}|\le 4d C^{-1}|\partial ^\prime A|.
\end{align*}
Towards the second inequality, note that 
\begin{align*}
    |A_{k+1}\setminus A_k|\le \sum_{\substack{c\in\mathcal{C}_{k+1}\colon c\, \cap  A_{k+1}\setminus A_k\neq \emptyset  } }|c|= \sum_{\substack{c\in\mathcal{C}_{k+1}\colon c\,  \cap  A_{k+1}\setminus A_k\neq \emptyset  } }2^{(k+1)d}.
\end{align*}
But for every such $c\in \mathcal{C}_{k+1}$ as in the sum, we know that $c$ is admissible and we can find among the cubes in $\mathcal{C}_k$ that constitute $c$ a $\tilde c\in \mathcal{C}_k$ that is admissible. Furthermore, since $c$ is not a subset of $A_k$,  among the cubes in $\mathcal{C}_k$ that constitute $c$, there is also a $ c^\prime\in \mathcal{C}_k$ that is not admissible. In particular, we can choose adjacent cubes $c_i,c_j\in \mathcal{C}_k$ such that $c_i$ is admissible and $c_j$ is not, while still $c_i,c_j\subset c$. By Lemma~\ref{Lemma:CoarseGrainedBoundary},
\begin{align*}
     |A_{k+1}\setminus A_k|&\le 2^{d+k}\sum_{\{i,j\}\subset \Z^d\colon i\overset{1}{\sim}j, \,c_i\text{ admissible}, \, c_j\text{ not }}2^{(d-1)k}\\
     &\le C^{-1} 2^{d+k}\sum_{\{i,j\}\subset \Z^d\colon i\overset{1}{\sim}j, \,c_i\text{ admissible}, \, c_j\text{ not }} |\partial^\prime A \cap \{\{x,y\}\setslash x\in c_i\cup c_j\}|\\
     &\le 4d C^{-1}2^{d+k}|\partial ^\prime A|,
\end{align*}
where $c_i$ and $c_j$ in the sum are always in $\mathcal{C}_k$.
Analogously, we can write
\begin{align*}
 |A_{k}\setminus A_{k+1}|\le\sum_{\substack{c\in\mathcal{C}_{k+1}\colon c\, \cap  A_{k}\setminus A_{k+1}\neq \emptyset  } }2^{(k+1)d}  
\end{align*}
and find, for every $c$ in the sum, two adjacent cubes in $\mathcal{C}_k$ that are subsets of $c$, where one is admissible and the other one is not. By the same argument as above
\begin{align*}
    |A_{k}\setminus A_{k+1}|\le 4d C^{-1} 2^{d+k} |\partial ^\prime A|
\end{align*}
and the second claim follows. Finally, again by the isoperimetric inequality \eqref{Eq:IsoperimetricInequality} there is $C' >0$ such that
\begin{align*}
    |A_k|\le 2|A|\le C'|\partial ^\prime A|^{d/(d-1)},
\end{align*}
and the proof is finished.
\end{proof}
As mentioned in the beginning of this section, we also need a bound on $M_{n,k}$ defined in~\eqref{Eq:DefinitionMnk}. 
\begin{lemma}\label{Lemma:CountingCoarseGrains}
There exists $C>0$ such that for all $n\in \Z_{>0}$ and $1\le k\le N(n)$
\begin{align*}
    M_{n,k}\le \exp(Ckn2^{-(d-1)k}).
\end{align*}
\end{lemma}
\begin{proof}In this proof the constant $C$ may change from line to line. Let $B$ be a coarse graining $B=A_k$ of some $A\in \mathfrak{A}_m$ with $|\partial^\prime A|=n$. Define the \textit{directed edge boundary} of $B$ as 
\begin{align*}
    \mathcal{E}(B)=\{(i,j)\in \Z^{2d}\setslash c_i\subset B, \, c_j\cap B=\emptyset,\, i\overset{1}{\sim}j\}
\end{align*}
and note that $\mathcal{E}$ is injective on $\{B \colon \exists A\in \mathfrak{A}_m,  |\partial^\prime A|=n,\, B=A_{k} \}$. Indeed, if $\mathcal{E}(B)=\mathcal{E}(\tilde B)$ and $B\neq \emptyset$, we can find $(i,j)\in \mathcal{E}(B)$. For every other $\tilde i \in \Z^d$, we can check if $c_{\tilde i}\subset B$ or $c_{\tilde i}\cap B=\emptyset $ by following any $1$-connected path $\pi$ from $\tilde i$ to $i$ and checking if $(\pi(t+1),\pi(t))\in \mathcal{E}(B)$ for some $t$ while $(\pi(s),\pi(s+1))\notin \mathcal{E}(B)$ for all $s<t$. Since $\mathcal{E}(B)=\mathcal{E}(\tilde B)$, this implies that  $c_{\tilde i}\subset B$ if and only if $c_{\tilde i}\subset \tilde B$. 

Next, let $\alpha (B)\in \Z_{\ge 0}$ and $\gamma_1(B),\ldots, \gamma_{\alpha(B)}(B)\Subset \Z^d$ as well as $\mathfrak{n}_k(B)\colon \gamma_k(B)\rightarrow \{A\Subset\Z^d,\, A\neq \emptyset\}$ for $k=1,\ldots, \alpha(B)$ such that $\gamma_1(B),\ldots, \gamma_{\alpha(B)}(B)$ are $1$-connected, non-empty and pairwise different and such that
\begin{align*}
    \mathcal{E}(B)=\bigcup_{k=1}^{\alpha(B)}\{(i,j)\setslash i \in \gamma_k(B),\, j\in \mathfrak{n}_k(i)\}.
\end{align*}
Overall, 
\begin{align*}
    B\mapsto \big\{\big(\gamma_1(B),\mathfrak{n}_1(B)\big),\ldots, \big(\gamma_{\alpha(B)}(B),\mathfrak{n}_{\alpha(B)}(B)\big)\big\} 
\end{align*}
is injective on $\{B \colon \exists A\in \mathfrak{A}_m,  |\partial^\prime A|=n,\, B=A_{k} \}$ and it suffices to bound the number of possible $\gamma_1(B),\ldots, \gamma_{\alpha(B)}(B),\mathfrak{n}_1(B),\ldots, \mathfrak{n}_{\alpha(B)}(B)$ that are attainable for a $B$ that is a coarse-graining of some set $A$. By the first inequality of Lemma~\ref{Lemma:Bounds_on_A_k}, there exists $C>0$ such that 
\begin{align}\label{Eq:Countingboundary}
    Cn\ge |\partial ^\prime B|\ge \sum_{k=1}^{\alpha(B)}|\gamma_k(B)| 2^{(d-1)k},
\end{align}
which means in particular $\alpha(B)\le \lfloor Cn2^{-(d-1)k}\rfloor\eqcolon \alpha_{n,k}$. Also, we know that $\|\boldsymbol{\ell}_\alpha\|_1\le \alpha_{n,k}$ when denoting $\ell_k=| \gamma_k(B)|$ and $\boldsymbol{\ell}_\alpha=(\ell_1,\ldots, \ell_\alpha)$. To bound the number of possible $\gamma_k(B)$, we can fix an arbitrary $i_k(B)\in \gamma_k(B)$ and then find 
\begin{align*}
    x_k(B)\in \partial^1 A\cap \Big(c_{i_k(B)}\cup \bigcup_{j\in \mathfrak{n}_k(i_k(B))}c_j\Big).
\end{align*}
Indeed, by definition of $\gamma_k(b)$ we have that $c_{i_k(B)}$ is admissible w.r.t.~$A$ and there is $j\in \mathfrak{n}_k(i_k(B))$ such that $c_j$ is not admissible w.r.t.~$A$. In particular, there are $x,y\in c_{i_k(B)}\cup c_j$ such that $x\in A$ but $y\notin A$. By following some $1$-connected path within $c_{i_k(B)}\cup c_j$ from $x$ to $y$, we can find $x_k(B)\in \partial ^1 A$. By Lemma~\ref{Lemma:ConnectivityBoundary}, we know that $\partial^1 A$ is $d$-connected. By Lemmas~\ref{Lemma:permutationBoundary}~and~\ref{Lemma:Equivalence_Boundaries}, after  reordering the indices in $\gamma_1(B),\ldots, \gamma_{\alpha(B)}(B)$, if necessary, we can assume without loss of generality that
\begin{align}\label{Eq:BoundingDifferencexj}
    \sum_{i=2}^{\alpha(B)}\| x_i(B)- x_{i-1}(B)\|_1\le Cn.
\end{align}
Finally, the number of possible $\mathfrak{n}_1(B),\ldots ,\mathfrak{n}_{\alpha(B)}(B)$ is bounded by $\exp(C\alpha_{n,k})$.
Overall, with the notation $\boldsymbol{d}_\alpha=(d_1,\ldots, d_\alpha)$ and $\boldsymbol{x}_\alpha=( x_1,\ldots,  x_\alpha)$, 
\begin{align*}
    M_{n,k}&\le \exp(C\alpha_{n,k})\sum_{\alpha =1}^{\alpha_{n,k}} \quad \sum_{\boldsymbol{\ell}_\alpha\colon \|\boldsymbol{\ell}_\alpha\|_1\le \alpha_{n,k}}\quad\sum_{\boldsymbol{d}_\alpha\colon  \|\boldsymbol{d}_\alpha\|_1\le C n}\quad \sum_{\boldsymbol{x}_\alpha\colon \| x_j- x_{j-1}\|_1 = d_j\, \forall j,\, \|x_1\|_{1}\le n} \\
    &\hspace{3cm}\prod_{j=1}^{\alpha} |\{\gamma\subset \Z^d\setslash \gamma \text{ is $d$-connected},\exists i \in \gamma, \,\exists j\colon \|i-j\|_1\le 1,\, x_j\in c_j \text{ and }|\gamma|=\ell_j \}|.
\end{align*}
By Lemma~\ref{Lemma:CountingConnectedSets},
\begin{align*}
    M_{n,k}\le \exp(C\alpha_{n,k})\alpha_{n,k}\, |\{\boldsymbol{\ell}_{\alpha_{n,k}}\colon \|\boldsymbol{\ell}_{\alpha_{n,k}}\|_1\le n\}|\, \sum_{\boldsymbol{d}_{\alpha_{n,k}}\colon \|\boldsymbol{d}_{\alpha_{n,k}}\|_1\le Cn}|\{\boldsymbol{x}_{\alpha_{n,k}}\colon \|x_j-x_{j-1}\|_1\le d_j,\, \|x_1\|_{1}\le n\}|.
\end{align*}
By a stars and bars argument and by Stirling's formula,
\begin{align}\label{Eq:BarsAndStars}
    |\{\boldsymbol{\ell}_{\alpha_{n,k}}\colon \|\boldsymbol{\ell}_{\alpha_{n,k}}\|_1=n\}|\le \frac{ n^{\alpha_{n,k}}}{\alpha_{n,k}!}\le \exp(C\alpha_{n,k}\log(n/\alpha_{n,k}))\le \exp(Ck\alpha_{n,k}).
\end{align}
Here and in the following we use that 
\begin{align*}
    \alpha_{n,k}\ge \alpha_{n,N(n)}\ge C n 2^{-(d-1)N(n)}\ge Cn^{1-(d-1)/d} =C n^{1/d}
\end{align*}
to neglect lower order growth in $n$ by using a larger constant $C$. Observe that
\begin{align*}
    |\{\boldsymbol{x}_{\alpha_{n,k}}\colon \|x_j-x_{j-1}\|_1\le d_j,\, \|x_1\|_{1}\le n\}|\le n ^dC^{\alpha_{n,k}}\prod_{j=2}^{\alpha_{n,k}}d_j^d.
\end{align*}
The right-hand side is maximal under the condition $\|\boldsymbol{d}_{\alpha_{n,k}}\|_1\le Cn$ for $d_i=Cn/\alpha_{n,k}$ for all $i$. Thus,
\begin{align*}
\sum_{\boldsymbol{d}_{\alpha_{n,k}}\colon \|\boldsymbol{d}_{\alpha_{n,k}}\|_1\le Cn}|\{\boldsymbol{x}_{\alpha_{n,k}}\colon \|x_i-x_{i-1}\|_1\le d_i\}|&\le |\{\boldsymbol{d}_{\alpha_{n,k}}\colon \|\boldsymbol{d}_{\alpha_{n,k}}\|_1\le C  n\}|\, \exp(C \alpha_{n,k}\log(n/\alpha_{n,k}))\\
&\le \exp(Ck\alpha_{n,k}),
\end{align*}
where we used again \eqref{Eq:BarsAndStars} for the last estimate. Overall, we proved 
$M_{n,k}\le \exp(Ck\alpha_{n,k})$,
as required.
\end{proof}
Note that the above lemma is the only place in the proof of Proposition~\ref{Prop:CoarseGraining}, more precisely \eqref{Eq:BoundingDifferencexj}, where we use that $A$ is $m$-simply connected. This is the reason why the argument is so flexible in the choice of connectedness.
\begin{proof}[Proof~of~Proposition~\ref{Prop:CoarseGraining}]
    We use the third inequality in Lemma~\ref{Lemma:Bounds_on_A_k} to bound the first sum in \eqref{Eq:Estimate} and the second inequality in Lemma~\ref{Lemma:Bounds_on_A_k} for the second sum. Furthermore, using Lemma~\ref{Lemma:CountingCoarseGrains} we arrive at
    \begin{align*}
    P&\le \sum_{n=1}^\infty M_{n,{N(n)}}\exp\Big(-\frac{cz^2n^{2-d/(d-1)}}{N(n)^4\varepsilon^2} \Big)+\sum_{n=1}^\infty \sum_{k=1}^{N(n)-1} M_{n,k}\exp\Big(-\frac{cz^2n}{k^4\varepsilon^2 2^k}\Big)\\
    &\le \sum_{n=1}^\infty \exp\Big(CN(n)n2^{-(d-1)N(n)}-\frac{cz^2n^{2-d/(d-1)}}{N(n)^4\varepsilon^2} \Big)+ N(n)\exp\Big(CN(n)n2^{-(d-1)N(n)}-\frac{cz^2n}{N(n)^4\varepsilon^2 2^{N(n)}}\Big)             
    \end{align*}
    for some $c>0$.
    By the choice of $N(n)$ in \eqref{Eq:ChoiceN},\
    \begin{align*}
        2^{-(d-1)N(n)}\le 2^{d-1} n^{-(d-1)/d}\quad \text{and}\quad 2^{-N(n)}\le 2 n^{-1/d}.
    \end{align*}
    Furthermore for $d\ge 3$, it holds that
    \begin{align*}
       2-d/(d-1)=1-1/(d-1)> 1/d \quad \text{and}\quad  1-1/d> 1/d=1-(d-1)/d.
    \end{align*}
    Note that this is the only place in the whole argument, where we use $d\ge 3$. Using both of the above and that $N(n)$ itself grows logarithmically, the negative terms in the exponentials dominate the positive ones and there are $\tilde C, C',\tilde c >0$ such that
    \begin{align*}
        P&\le  C'\sum_{n=1}^\infty \exp(-cz^2n^{2-d/(d-1)}/(N(n)^4\varepsilon^2) )+ \exp(- cz^2 n^{(d-1)/d}/(N(n)^4\varepsilon^2))\le\tilde  C\exp(-\tilde c z^2/\varepsilon^2),
    \end{align*}
    as desired. 
\end{proof}

\section*{References}
\renewcommand*{\bibfont}{\footnotesize}
\printbibliography[heading = none]


\end{document}